\documentclass[11pt]{article}

\usepackage[sort,numbers]{natbib}
\usepackage[bookmarks=true,bookmarksnumbered=true,colorlinks=true,allcolors=blue]{hyperref}
\usepackage{amsmath,amssymb,amsfonts,amsthm}
\usepackage{lscape}
\usepackage{arydshln}
\renewcommand*{\baselinestretch}{1.25}

\usepackage{geometry}
\usepackage{indentfirst}

\usepackage{enumitem,color}

\newtheorem{theorem}{Theorem}[section]
\newtheorem{lemma}{Lemma}[section]
\newtheorem{proposition}{Proposition}[section]
\newtheorem{corollary}{Corollary}[section]
\theoremstyle{definition}

\newtheorem*{rmk*}{Remark}
\newtheorem{rmk}{Remark}[section]

\renewcommand*\proofname{\upshape{\bfseries{Proof}}}

\allowdisplaybreaks

\numberwithin{equation}{section}

\makeatletter
    \renewcommand*{\section}{\@startsection{section}{1}{\z@}%
    {6pt}{3pt}{\reset@font\normalsize\bfseries}}
\makeatother
    
\makeatletter
    \renewcommand*{\subsection}{\@startsection{subsection}{2}{\z@}%
    {3pt}{3pt}{\reset@font\normalsize\itshape}}
\makeatother

\makeatletter
    \renewcommand*{\subsubsection}{\@startsection{subsubsection}{3}{\z@}%
    {3pt}{3pt}{\reset@font\normalsize\mdseries\itshape}}
\makeatother

\makeatletter
\def\@listi{\leftmargin\leftmargini
  \topsep=.5\baselineskip 
  \partopsep=0pt \parsep=0pt \itemsep=0pt}
\let\@listI\@listi
\@listi
\def\@listii{\leftmargin\leftmarginii
  \labelwidth\leftmarginii \advance\labelwidth-\labelsep
  \topsep=0pt \partopsep=0pt \parsep=0pt \itemsep=0pt}
\def\@listiii{\leftmargin\leftmarginiii
  \labelwidth\leftmarginiii \advance\labelwidth-\labelsep
  \topsep=0pt \partopsep=0pt \parsep=0pt \itemsep=0pt}
\def\@listiv{\leftmargin\leftmarginiv
  \labelwidth\leftmarginiv \advance\labelwidth-\labelsep
  \topsep=0pt \partopsep=0pt \parsep=0pt \itemsep=0pt}
\makeatother

\makeatletter
\newcommand{\opnorm}{\@ifstar\@opnorms\@opnorm}
\newcommand{\@opnorms}[1]{%
  \left|\mkern-1.5mu\left|\mkern-1.5mu\left|
   #1
  \right|\mkern-1.5mu\right|\mkern-1.5mu\right|
}
\newcommand{\@opnorm}[2][]{%
  \mathopen{#1|\mkern-1.5mu#1|\mkern-1.5mu#1|}
  #2
  \mathclose{#1|\mkern-1.5mu#1|\mkern-1.5mu#1|}
}
\makeatother

\makeatletter
\renewenvironment{proof}[1][\proofname]{\par
  \pushQED{\qed}%
  \normalfont \topsep6\p@\@plus6\p@\relax
  \trivlist
  \item[\hskip\labelsep
        \bfseries
    #1\@addpunct{.}]\ignorespaces
}{%
  \popQED\endtrivlist\@endpefalse
}
\makeatother

\DeclareMathOperator{\vectorize}{vec}

\DeclareMathOperator{\trace}{tr}
\DeclareMathOperator{\diag}{diag}
\DeclareMathOperator{\sign}{sign}
\DeclareMathOperator{\expectation}{E}

\DeclareMathOperator{\bic}{BIC}
\DeclareMathOperator{\leb}{\mathsf{Leb}}

\newcommand{\labs}{\left|}
\newcommand{\rabs}{\right|}
\newcommand{\lpa}{\left(}
\newcommand{\rpa}{\right)}

\newcommand{\bs}[1]{\boldsymbol{#1}}
\newcommand{\ol}[1]{\overline{#1}}
\newcommand{\ul}[1]{\underline{#1}}
\newcommand{\wh}[1]{\widehat{#1}}
\newcommand{\wt}[1]{\widetilde{#1}}
\newcommand{\ex}[1]{\expectation\left[#1\right]}
\newcommand{\expe}[1]{\mathrm{E}[#1]}
\newcommand{\mf}[1]{\mathfrak{#1}}
\newcommand{\mcl}[1]{\mathcal{#1}}

\newcommand{\pck}[1]{{\bfseries#1}}

\makeatletter
\newcommand{\pushright}[1]{\ifmeasuring@#1\else\omit\hfill$\displaystyle#1$\fi\ignorespaces}
\newcommand{\pushleft}[1]{\ifmeasuring@#1\else\omit$\displaystyle#1$\hfill\fi\ignorespaces}
\makeatother

\title{De-biased graphical Lasso for high-frequency data}
\author{Yuta Koike\thanks{Mathematics and Informatics Center and Graduate School of Mathematical Sciences, The University of Tokyo, 3-8-1 Komaba, Meguro-ku, Tokyo 153-8914 Japan}
\thanks{CREST, Japan Science and Technology Agency}
}

\begin{document}

\maketitle

\begin{abstract}

This paper develops a new statistical inference theory for the precision matrix of high-frequency data in a high-dimensional setting. The focus is not only on point estimation but also on interval estimation and hypothesis testing for entries of the precision matrix. 
To accomplish this purpose, we establish an abstract asymptotic theory for the weighted graphical Lasso and its de-biased version without specifying the form of the initial covariance estimator. 
We also extend the scope of the theory to the case that a known factor structure is present in the data. 
The developed theory is applied to the concrete situation where we can use the realized covariance matrix as the initial covariance estimator, and we obtain a feasible asymptotic distribution theory to construct (simultaneous) confidence intervals and (multiple) testing procedures for entries of the precision matrix. 
%
\vspace{3mm}

\noindent \textit{Keywords}: asymptotic mixed normality; factor model; high-dimensions; Malliavin calculus; precision matrix; sparsity.

\end{abstract}

\section{Introduction}

In high-frequency financial econometrics, covariance matrix estimation of asset returns has been extensively studied in the past two decades. 
High-frequency financial data are commonly modeled as a discretely observed semimartingale for which the quadratic covariation matrix plays the role of the covariance matrix, so their treatments are often different from those in a standard i.i.d.~setting. 
In the recent years, motivated by application to portfolio allocation and risk management in a large scale asset universe, the high-dimensionality problem has attracted much attention in this area. 
Since the 2000s, great progress has been made in high-dimensional covariance estimation from i.i.d.~data, so researchers are naturally led to apply the techniques developed therein to the context of high-frequency data. 
For example, \citet{WZ2010} have applied the entry-wise shrinkage methods considered in \cite{BL2008th,BL2008band} to estimating the covariance matrix of high-frequency data which are asynchronously observed with noise. See also \cite{TWZ2013,TWC2013,KWZ2016,KKLW2018} for further developments in this approach. 
In the meantime, it is well-recognized that the \textit{factor structure} is an important ingredient both theoretically and empirically for financial data. In the context of high-dimensional covariance estimation from high-frequency data, this perspective was first taken into account by \citet{FFX2016} and subsequently built up by, among others, \cite{AX2017,FK2017,DLX2019}. 
Other common methods used in i.i.d.~settings have also been investigated in the literature of high-frequency financial econometrics. 
\citet{HKO2012} and \citet{MN2017} formally apply eigenvalue regularization methods based on random matrix theory to high-frequency data. 
\citet{LFH2017} accommodate the non-linear shrinkage estimator of \cite{LW2012} to a high-frequency data setting with the help of the spectral distribution theory for the realized covariance matrix developed in \cite{ZL2011}. 
\citet{BNS2018} employ the $\ell_1$-penalized Gaussian MLE, which is known as the \textit{graphical Lasso}, to estimate the precision matrix (the inverse of the covariance matrix) of high-frequency data. The latter approach is closely related to the methodology we will focus on. 

Despite the recent advances in this topic as above, most studies in this area focus only on \textit{point estimation} of covariance and precision matrices, and there are little work about \textit{interval estimation} and \textit{hypothesis testing} for these objects. 
A few exceptions are \cite{KL2018,Pelger2019} and \cite{Koike2018sk}. 
The first two articles are concerned with continuous-time factor models: \citet{KL2018} propose a test for the constancy of the factor loading matrix, while \citet{Pelger2019} assumes constant loadings and develops an asymptotic distribution theory to make inference for the factors and loadings. Meanwhile, \citet{Koike2018sk} establishes a high-dimensional central limit theorem for the realized covariance matrix which allows us to construct simultaneous confidence regions or carry out multiple testing for entries of the high-dimensional covariance matrix of high-frequency data. 
The aim of this study is to develop such a statistical inference theory for the \textit{precision matrix} of high-frequency data. This is naturally motivated by the fact that the precision matrix of asset returns plays an important role in mean-variance analysis of portfolio allocation (see e.g.~\cite[Chapter 5]{Cochrane2005}). 
We accomplish this purpose by imposing a sparsity assumption on the precision matrix. Such an assumption has a clear interpretation in connection with \textit{Gaussian graphical models}: For a Gaussian random vector $\bs{\xi}=(\xi_1,\dots,\xi_d)^\top$ with covariance matrix $\Sigma$, $\xi_i$ and $\xi_j$ are conditionally independent given the other components if and only if the $(i,j)$-th entry of $\Sigma^{-1}$ is equal to 0, so the sparsity of $\Sigma^{-1}$ is interpreted as the sparsity of the edge structure of the \textit{conditional independence graph} associated with $\bs{\xi}$. We refer to Chapter 13 of \cite{BvdG2011} and references therein for more details on graphical models. 
This standpoint also makes it interesting to estimate the precision matrix of financial data in view of the recent attention to financial network analysis such as \cite{ACOT2012}. 

Statistical inference for high-dimensional sparse precision matrices has been actively studied in the recent literature, and various methodologies have ever been proposed; see \cite{JvdG2018} for an overview. Among others, this paper studies (a weighted version of) the de-biased (or de-sparsified) graphical Lasso in the context of high-frequency data. The de-biased graphical Lasso was introduced in \citet{JvdG2015} and its theoretical property was investigated in the i.i.d.~case. In this paper we consider its weighted version discussed in \cite{JvdG2018} because of its theoretically preferable behavior (see Remark \ref{rmk:weighted}). 
Compared to the i.i.d.~case, we need to handle a new theoretical difficulty in the application to high-frequency data, which is caused by the non-ergodic nature of the problem. That is, the precision matrix of high-frequency data is generally stochastic and not (stochastically) independent of the observation data. 
In our context, the precision matrix appears in the coefficients of the linear approximation of the de-biased estimator (see Lemma \ref{prop:AL}), so it spoils the martingale structure of the linear approximation which we usually have in the i.i.d.~case. In a low-dimensional setting, this issue is typically resolved by the concept of \textit{stable convergence} (see e.g.~\cite{PV2010}), but the applicability of this approach is questionable in our setting due to the high-dimensionality (see pages 1451--1452 of \cite{Koike2018sk} for a discussion). Instead, we rely on the recent high-dimensional central limit theory of \cite{Koike2018sk} to establish the asymptotic distribution theory for the de-biased estimator, where we settle the above difficulty with the help of Malliavin calculus. 

The rest of this paper is organized as follows. 
In Section \ref{sec:main} we develop an abstract asymptotic theory for the weighted graphical Lasso based on a generic estimator for the quadratic covariation matrix of a high-dimensional semimartingale. This allows us to flexibly apply the developed theory to various settings arising in high-frequency financial econometrics. In Section \ref{sec:factor} we extend the scope of the theory to a situation where a known factor structure is present in data and a sparsity assumption is imposed on the precision matrix of the residual process rather than that of the original process. 
In Section \ref{sec:rc}, we apply the abstract theory developed in Section \ref{sec:factor} to a concrete setting where we observe the process at equidistant times without jumps and noise. 
Section \ref{sec:simulation} conducts a Monte Carlo study to assess the finite sample performance of the asymptotic theory. 
All the technical proofs are collected in the Appendix. 

\section*{Notation}

Throughout the paper, we assume $d\geq2$. 
$\top$ stands for the transpose of a matrix. For a vector $x\in\mathbb{R}^d$, we write the $i$-th component of $x$ as $x^i$ for $i=1,\dots,d$. 
For two vectors $x,y\in\mathbb{R}^d$, the statement $x\leq y$ means $x^i\leq y^i$ for all $i=1,\dots,d$. 
The identity matrix of size $d$ is denoted by $\mathsf{E}_d$. 
We write $\mathbb{R}^{l\times k}$ for the set of all $l\times k$ matrices. 
$\mcl{S}_d$ denotes the set of all $d\times d$ symmetric matrices. 
$\mcl{S}_d^+$ denotes the set of all $d\times d$ positive semidefinite matrices. 
$\mcl{S}_d^{++}$ denotes the set of all $d\times d$ positive definite matrices. 
For a $l\times k$ matrix $A$, the $(i,j)$-th entry of $A$ is denoted by $A^{ij}$. Also, $A^{i\cdot}$ and $A^{\cdot j}$ denote the $i$-th row vector and the $j$-th column vector, respectively (both are regarded as column vectors). 
We write $\vectorize(A)$ for the vectorization of $A$:
\[
\vectorize(A):=(A^{11},\dots,A^{l1},A^{12},\dots,A^{l2},\dots,A^{1k},\dots,A^{lk})^\top\in\mathbb{R}^{lk}.
\]
For every $w\in[1,\infty]$, we set 
\[
\|A\|_{\ell_w}:=
\left\{\begin{array}{ll}
\{\sum_{i=1}^l\sum_{j=1}^{k}|A^{ij}|^w\}^{1/w} & \text{if }w<\infty,\\
\max_{1\leq i\leq l}\max_{1\leq j\leq k}|A^{ij}| & \text{if }w=\infty.
\end{array}\right.
\] 
Also, we write $\opnorm{A}_w$ for the $\ell_w$-operator norm of $A$:
\[
\opnorm{A}_w:=\sup\{\|Ax\|_{\ell^w}:x\in\mathbb{R}^k,\|x\|_{\ell_w}=1\}.
\]
It is well-known that $\opnorm{A}_1=\max_{1\leq j\leq k}\sum_{i=1}^{l}|A^{ij}|$ and $\opnorm{A}_\infty=\max_{1\leq i\leq l}\sum_{j=1}^k|A^{ij}|$. 
When $l=k$, $\diag(A)$ denotes the diagonal matrix with the same diagonal entries as $A$, and we set $A^-:=A-\diag(A)$. If $A$ is symmetric, we denote by $\Lambda_{\max}(A)$ and $\Lambda_{\min}(A)$ the maximum and minimum eigenvalues of $A$, respectively. 
For two matrices $A$ and $B$, $A\otimes B$ denotes their Kronecker product. When $A$ and $B$ has the same size, we write $A\circ B$ for their Hadamard product. 

For a random variable $\xi$ and $p\in(0,\infty]$, $\|\xi\|_p$ denotes the $L^p$-norm of $\xi$. 
For a $l$-dimensional semimartingale $X=(X_t)_{t\in[0,1]}$ and a $k$-dimensional semimartingale $Y=(Y_t)_{t\in[0,1]}$, we define $\Sigma_{XY}:=[X,Y]_1:=([X^i,Y^j]_1)_{1\leq i\leq l,1\leq j\leq k}$. We write $\Sigma_X=\Sigma_{XX}$ for short. If $\Sigma_X$ is a.s.~invertible, we write $\Theta_X:=\Sigma_X^{-1}$. 

\section{Estimators and abstract results}\label{sec:main}

Given a stochastic basis $\mcl{B}=(\Omega,\mcl{F},(\mcl{F}_t)_{t\in[0,1]},P)$, we consider a $d$-dimensional semimartingale $Y=(Y_t)_{t\in[0,1]}$ defined there. We assume $\Sigma_Y=[Y,Y]_1$ is a.s.~invertible. 
In this paper we consider the asymptotic theory such that the dimension $d$ possibly depends on a parameter $n\in\mathbb{N}$ so that $d=d_n\to\infty$ as $n\to\infty$. As a consequence, both $\mcl{B}$ and $Y$ may also depend on $n$. However, following the custom of the literature, we omit the indices $n$ from these objects and many other ones appearing below. 

Our aim is to estimate the precision matrix $\Theta_Y=\Sigma_Y^{-1}$ when we have an estimator $\hat{\Sigma}_n$ for $\Sigma_Y$; as a corollary, we can also estimate $\Sigma_Y$ itself. 
We assume that $\hat{\Sigma}_n$ is an $\mcl{S}_d^+$-valued random variable all of whose diagonal entries are a.s.~positive, but we do not specify the form of $\hat{\Sigma}_n$ because the asymptotic theory developed in this section depends on the property of $\hat{\Sigma}_n$ rather than their construction. 
This is convenient because construction of the estimator depends heavily on observation schemes for $Y$ (with or without noise, synchronous or not, continuous or discontinuous and so on; see \cite{KY2019} for details). In Section \ref{sec:rc} we illustrate how we apply the abstract theory developed in this and the next sections to a concrete situation. 

We use the \textit{weighted graphical Lasso} to estimate $\Theta_Y$ (cf.~\cite{JvdG2018}). The weighted graphical Lasso estimator $\hat{\Theta}_\lambda$ with penalty parameter $\lambda>0$ based on $\hat{\Sigma}_n$ is defined by
\begin{equation}\label{wglasso}
\hat{\Theta}_\lambda:=\arg\min_{\Theta\in\mcl{S}_d^{++}}\left\{\trace\left(\Theta \hat{\Sigma}_n\right)-\log\det\left(\Theta\right)+\lambda\sum_{i\neq j}\hat{ V}_n^{ii}\hat{ V}_n^{jj}\left|\Theta^{ij}\right|\right\},
\end{equation}
where $\hat{ V}_n:=\diag(\hat{\Sigma}_n)^{\frac{1}{2}}$. 
According to the proof of \cite[Lemma 1]{DGK2008}, the optimization problem in \eqref{wglasso} has the unique solution when $\lambda>0$ and $\hat{\Sigma}_n$ is positive semidefinite and all the diagonal entries of $\hat{\Sigma}_n$ are positive, so $\hat{\Theta}_\lambda$ is a.s.~defined in our setting. 
In the following we allow $\lambda$ to be a random variable because we typically select $\lambda$ in a data-driven way. 

%
To analyze the theoretical property of $\hat{\Theta}_\lambda$, it is convenient to consider the graphical Lasso estimator $\hat{K}_\lambda$ based on the correlation matrix estimator $\hat{R}_n:=\hat{ V}_n^{-1}\hat{\Sigma}_n\hat{ V}_n^{-1}$ as follows:
\begin{equation}\label{glasso-cormat}
\hat{K}_\lambda:=\arg\min_{K\in\mcl{S}_d^{++}}\left\{\trace\left(K \hat{R}_n\right)-\log\det\left(K\right)+\lambda\left\|K^{-}\right\|_{\ell_1}\right\}.
\end{equation}
We can easily check $\hat{\Theta}_\lambda=\hat{ V}_n^{-1}\hat{K}_\lambda\hat{ V}_n^{-1}$. 
\begin{rmk}\label{rmk:weighted}
As pointed out in \citet{RBLZ2008} and \citet{JvdG2018}, the graphical Lasso based on correlation matrices is theoretically preferable to that based on covariance matrices (so the weighted graphical Lasso is also preferable). In particular, we do not need to impose the so-called \textit{irrepresentability condition} on $\Sigma_Y$ to derive the theoretical properties of our estimators, which contrasts with \citet{BNS2018} (see Assumption 2 in \cite{BNS2018}).  
\end{rmk}

We introduce some notation related to the sparsity assumptions we will impose on $\Theta_Y$. Let $A\in\mcl{S}_d$. For $j=1,\dots,d$, we set $\mf{D}_j(A):=\{i:A^{ij}\neq0,i\neq j\}$ and $\mf{d}_j(A):=\#\mf{D}_j(A)$. Then we define $\mf{d}(A):=\max_{1\leq j\leq d}\mf{d}_j(A)$. We also define $S(A):=\bigcup_{j=1}^d\mf{D}_j(A)=\{(i,j):A^{ij}\neq0,i\neq j\}$ and $s(A):=\# S(A)$. 
These quantities have a clear interpretation when the matrix $A$ represents the edge structure of some graph so that $A^{ij}\neq0$ is equivalent to the presence of an edge between vertices $i$ and $j$ for $i\neq j$; in this case, $\mf{d}_j(A)$ is the number of edges adjacent to vertex $j$ (which is called the \textit{degree} of vertex $j$) and $s(A)$ is the total number of edges contained in the graph. 

To derive our asymptotic results, we will impose the following structural assumptions on $\Sigma_Y$. 
\begin{enumerate}[label={\normalfont[A\arabic*]}]


\item\label{ass:eigen} $\Lambda_\mathrm{max}(\Sigma_Y)+1/\Lambda_\mathrm{min}(\Sigma_Y)=O_p(1)$ as $n\to\infty$.

\item\label{ass:sparse} $s(\Theta_Y)=O_p(s_n)$ as $n\to\infty$ for some sequence $s_n\in[1,\infty)$, $n=1,2,\dots$.

\item\label{ass:degree} $\mf{d}(\Theta_Y)=O_p(\mf{d}_n)$ as $n\to\infty$ for some sequence $\mf{d}_n\in[1,\infty)$, $n=1,2,\dots$. 

\end{enumerate}

\ref{ass:eigen} is standard in the literature; see e.g.~Condition A1 in \cite{JvdG2018}. 
\ref{ass:sparse} states that the sparsity of $\Theta_Y$ is controlled by the deterministic sequence $s_n$; we will require the growth rate of $s_n$ to be moderate. 
\ref{ass:degree} is another sparsity assumption on $\Theta_Y$. It is weaker than \ref{ass:sparse} in the sense that it always holds true with $\mf{d}_n=s_n$ under \ref{ass:sparse}. However, we can generally take $\mf{d}_n$ smaller than $s_n$.  

\subsection{Consistency}\label{sec:rate}

Set $ V_Y:=\diag(\Sigma_Y)^{\frac{1}{2}}$, $R_Y:= V_Y^{-1}\Sigma_Y V_Y^{-1}$ and $K_Y:=R_Y^{-1}$. 
\begin{proposition}\label{glasso:rate}
Assume \ref{ass:eigen}--\ref{ass:sparse}. Let $(\lambda_n)_{n=1}^\infty$ be a sequence of positive-valued random variables satisfying the following conditions:
\begin{enumerate}[label={\normalfont[B\arabic*]}]

\item\label{ass:est} $\lambda_n^{-1}\|\hat{\Sigma}_n-\Sigma_Y\|_{\ell_\infty}\to^p0$ as $n\to\infty$.

\item\label{ass:rate} $s_n\lambda_n\to^p0$ as $n\to\infty$.

\end{enumerate}
Then we have
\begin{equation}\label{rate:K}
\lambda_n^{-1}\|\hat{K}_{\lambda_n}-K_Y\|_{\ell_2}=O_p(\sqrt{s_n}),\qquad
\lambda_n^{-1}\opnorm{\hat{K}_{\lambda_n}-K_Y}_w=O_p(s_n)
\end{equation}
and
\begin{equation}\label{rate:Theta}
\lambda_n^{-1}\opnorm{\hat{\Theta}_{\lambda_n}-\Theta_Y}_w=O_p(s_n),\qquad
\lambda_n^{-1}\opnorm{\hat{\Theta}_{\lambda_n}^{-1}-\Sigma_Y}_2=O_p(s_n)
\end{equation}
as $n\to\infty$ for any $w\in[1,\infty]$. 
\end{proposition}

Proposition \ref{glasso:rate} is essentially a rephrasing of Theorem 14.1.3 in \cite{JvdG2018}. To get a better convergence rate in Proposition \ref{glasso:rate}, we should choose $\lambda_n$ as small as possible, where a lower bound of $\lambda_n$ is determined by the convergence rate of $\hat{\Sigma}_n$ in the $\ell_\infty$-norm by \ref{ass:est}. One typically derives this convergence rate by establishing entry-wise concentration inequalities for $\hat{\Sigma}_n$. Such inequalities have already been established for various covariance estimators used in high-frequency financial econometrics; see Theorems 1--2 and Lemma 3 in \cite{FLY2012}, Theorem 1 in \cite{TWZ2013}, Theorem 1 in \cite{KW2016}, and Theorem 2 in \cite{BNS2018} for example. 
We however note that $\hat{\Sigma}_n$ should be positive semidefinite to ensure that the graphical Lasso has the unique solution. This property is not necessarily ensured by many covariance estimators used in this area. In this regard, we mention that pre-averaging and realized kernel estimators have versions to ensure this property, for which relevant bounds are available in \cite[Theorem 2]{KWZ2016} and \cite[Lemma 1]{DLX2019}. 
%
\begin{rmk}[Comparison to \citet{BNS2018}]
Compared with \cite[Theorem 1]{BNS2018}, Proposition \ref{glasso:rate} has two major theoretical improvements. 
First, Proposition \ref{glasso:rate} does not assume the so-called irrepresentability condition, which is imposed in \cite[Theorem 1]{BNS2018} as Assumption 2. 
Second, Proposition \ref{glasso:rate} gives consistency in the $\ell_w$-operator norm for all $w\in[1,\infty]$, while \cite[Theorem 1]{BNS2018} only shows consistency in the $\ell_\infty$-norm. 
We shall remark that consistency in matrix operator norms is important in application. For example, the consistency of $\hat{\Theta}_{\lambda_n}$ in the $\ell_2$-operator norm implies that eigenvalues of $\hat{\Theta}_{\lambda_n}$ consistently estimate the corresponding eigenvalues of $\Theta_Y$. Also, the consistency in the $\ell_\infty$-operator norm ensures $\|\hat{\Theta}_{\lambda_n}x-\Theta_Yx\|_{\ell_\infty}\to^p0$ as $n\to\infty$ for any $x\in\mathbb{R}^d$ such that $\|x\|_{\ell_\infty}=O(1)$. This result is important for portfolio allocation because the weight vector for the global minimum variance portfolio is given by $\Theta_Y\bs{1}/\bs{1}^\top\Theta_Y\bs{1}$ when assets have covariance matrix $\Sigma_Y$, where $\bs{1}=(1,\dots,1)^\top\in\mathbb{R}^d$; see e.g.~\cite[Section 5.2]{Cochrane2005}.  

On the other hand, unlike \cite[Theorem 1]{BNS2018}, we do not show selection consistency (i.e.~$P(S(\hat{\Theta}_{\lambda_n})=S(\Theta_Y))\to1$ as $n\to\infty$) under our assumptions. Indeed, in the linear regression setting, it is known that an irrepresentability type condition is \textit{necessary} for the selection consistency of the Lasso; see \cite[Section 7.5.3]{BvdG2011} for more details. 
However, we shall remark that the asymptotic mixed normality of the de-biased estimator stated below might be used to construct an estimator having selection consistency via thresholding; see \cite[Section 3.1]{RSZZ2015} and \cite[Section 4.2]{CQYZ2018} for such applications. 
\end{rmk}

\subsection{Asymptotic mixed normality}\label{sec:AMN}

The following lemma states that $\hat{\Theta}_{\lambda_n}-\Theta_Y$ is asymptotically linear in $\hat{\Sigma}_n-\Sigma_Y$ after bias correction when $\Theta_Y$ is sufficiently sparse. 
\begin{lemma}\label{prop:AL}
Suppose that the assumptions of Proposition \ref{glasso:rate} and \ref{ass:degree} are satisfied. Then we have
\[
\lambda_n^{-2}\|\hat{\Theta}_{\lambda_n}-\Theta_Y-\Gamma_n+\Theta_Y(\hat{\Sigma}_n-\Sigma_Y)\Theta_Y\|_{\ell_\infty}=O_p(s_n\sqrt{\mf{d}_n})
\]
as $n\to\infty$, where $\Gamma_n:=-(\hat{\Theta}_{\lambda_n}-\hat{\Theta}_{\lambda_n}\hat{\Sigma}_n\hat{\Theta}_{\lambda_n})$.
\end{lemma}
Lemma \ref{prop:AL} is an almost straightforward consequence of \eqref{rate:Theta} and the Karush-Kuhn-Tucker (KKT) conditions for the optimization problem in \eqref{wglasso}. 
As a consequence of this lemma, we obtain the following result, which states that the ``de-biased'' weighted graphical Lasso estimator $\hat{\Theta}_{\lambda_n}-\Theta_Y-\Gamma_n$ inherits the asymptotic mixed normality of $\hat{\Sigma}_n$. 

\begin{proposition}\label{prop:AMN}
Suppose that the assumptions of Lemma \ref{prop:AL} are satisfied. 
For every $n\in\mathbb{N}$, let $a_n>0$, $\mathfrak{C}_n$ be a $d^2\times d^2$ positive semidefinite random matrix and $J_n$ be an $m\times d^2$ random matrix, where $m=m_n$ may depend on $n$. 
Assume $a_n\opnorm{J_n}_\infty\lambda_n^2s_n\sqrt{\mf{d}_n\log (m+1)}\to^p0$ as $n\to\infty$.
Assume also that
\begin{equation}\label{est:AMN}
\lim_{n\to\infty}\sup_{y\in\mathbb{R}^m}\left|P\left(a_n\wt{J}_n\vectorize\left(\hat{\Sigma}_n-\Sigma_Y\right)\leq y\right)-P\left(\wt{J}_n\mathfrak{C}_n^{1/2}\zeta_n\leq y\right)\right|=0
\end{equation}
and
\begin{equation}\label{diag-tight}
\lim_{b\downarrow0}\limsup_{n\to\infty}P(\min\diag(\wt{J}_n\mathfrak{C}_n\wt{J}_n^\top)<b)=0
\end{equation}
as $n\to\infty$, where $\wt{J}_n:=-J_n(\Theta_Y\otimes\Theta_Y)$ and $\zeta_n$ is a $d^2$-dimensional standard Gaussian vector independent of $\mathcal{F}$, which is defined on an extension of the probability space $(\Omega,\mathcal{F},P)$ if necessary. 
Then,  
\[
\lim_{n\to\infty}\sup_{y\in\mathbb{R}^m}\left|P\left(a_nJ_n\vectorize\left(\hat{\Theta}_{\lambda_n}-\Gamma_n-\Theta_Y\right)\leq y\right)-P\left(\wt{J}_n\mathfrak{C}_n^{1/2}\zeta_n\leq y\right)\right|=0.
\]
\end{proposition}
In a standard i.i.d.~setting such that $\Theta_Y$ is non-random, we can usually verify \eqref{est:AMN} by classical Lindeberg's central limit theorem when $m=1$ and $J_n$ is non-random because $a_n\wt{J}_n\vectorize\left(\hat{\Sigma}_n-\Sigma_Y\right)$ can be written as a sum of independent random variables; see the proof of \cite[Theorem 1]{JvdG2015} for example. 
By contrast, $\Theta_Y$ is generally random and not independent of $\hat{\Sigma}_n-\Sigma_Y$ in our setting, so $a_n\wt{J}_n\vectorize\left(\hat{\Sigma}_n-\Sigma_Y\right)$ may not be a martingale even if $\vectorize\left(\hat{\Sigma}_n-\Sigma_Y\right)$ is a martingale. In the case that $d$ is fixed, we typically resolve this issue by proving \textit{stable} convergence in law of $\vectorize\left(\hat{\Sigma}_n-\Sigma_Y\right)$; see e.g.~\cite{PV2010} for details. However, extension of this approach to the case that $d\to\infty$ as $n\to\infty$ is far from trivial as discussed at the beginning of \cite[Section 3]{Koike2018sk}. For this reason, \cite{Koike2018sk} gives a result to directly establish \eqref{est:AMN} type convergence in a high-dimensional setting. This result will be used in Section \ref{sec:rc} to apply our abstract theory to a more concrete setting. 
\begin{rmk}
Proposition \ref{prop:AMN} also allows $m$ to diverge as $n\to\infty$, which is necessary when we need to derive an asymptotic approximation of the \textit{joint} distribution of $\vectorize\left(\hat{\Theta}_{\lambda_n}-\Gamma_n-\Theta_Y\right)$. Such an approximation can be used to make simultaneous inference for entries of $\Theta_Y$; see \cite{CQYZ2018} for example. 
\end{rmk}
%

\section{Factor structure}\label{sec:factor}

In financial applications, it is often important to take account of the factor structure of asset prices. In fact, many empirical studies have documented the existence of common factors in financial markets (e.g.~\cite[Section 6.5]{CLM1997}). Also, factor models play a dominant role in asset pricing theory (cf.~\cite[Chapter 9]{Cochrane2005}). 
When common factors are present across asset returns, the precision matrix cannot be sparse because all pairs of the assets are partially correlated given other assets through the common factors. Therefore, in such a situation, it is common practice to impose a sparsity assumption on the precision matrix of the residual process which is obtained after removing the co-movements induced by the factors (see e.g.~\cite[Section 4.2]{BNS2018} and \cite[Section 4.2]{BBL2018}). 
In this section we accommodate the theory developed in Section \ref{sec:main} to such an application. 

Specifically, suppose that we have an $r$-dimensional known factor process $X$, and consider the following continuous-time factor model:
\begin{equation}\label{factor-model}
Y=\beta X+Z.
\end{equation}
Here, $\beta$ is a non-random $d\times r$ matrix and $Z$ is a $d$-dimensional semimartingale such that $[Z,X]_1=0$. 
$\beta$ and $Z$ represent the factor loading matrix and residual process of the model, respectively. 
This model is widely used in high-frequency financial econometrics; see \cite{DLX2019,FFX2016,AX2017} in the context of high-dimensional covariance matrix estimation. One restriction of the model \eqref{factor-model} is that the factor loading $\beta$ is assumed to be constant, but there is empirical evidence that $\beta$ may be regarded as constant in short time intervals (one week or less); see \cite{RTT2015,KL2018} for instance. 

\begin{rmk}
The number of factors $r$ possibly depends on $n$ and (slowly) diverges as $n\to\infty$. Also, $\beta$ may depend on $n$.  
\end{rmk}

We are interested in estimating $\Sigma_Y$ based on observation data for $X$ and $Y$ while taking account of the factor structure given by \eqref{factor-model}. 
Suppose that we have generic estimators $\hat{\Sigma}_{Y,n},\hat{\Sigma}_{X,n}$ and $\hat{\Sigma}_{YX,n}$ for $\Sigma_Y,\Sigma_X$ and $\Sigma_{YX}$, respectively. $\hat{\Sigma}_{Y,n},\hat{\Sigma}_{X,n}$ and $\hat{\Sigma}_{YX,n}$ are assumed to be random variables taking values in $\mcl{S}_d,\mcl{S}_r^+$ and $\mathbb{R}^{d\times r}$, respectively. 
Now, by assumption we have
\begin{equation}\label{eq:factor-qc}
\Sigma_Y=\beta\Sigma_X\beta^\top+\Sigma_Z.
\end{equation}
Assume $\Sigma_X$ is a.s.~invertible. Then $\beta$ can be written as $\beta=\Sigma_{YX}\Sigma_X^{-1}$. 
Therefore, we can naturally estimate $\beta$ by $\hat{\beta}_n:=\hat{\Sigma}_{YX,n}\hat{\Sigma}_{X,n}^{-1}$, provided that $\hat{\Sigma}_{X,n}$ is invertible. 
In practical applications, the invertibility of $\hat{\Sigma}_{X,n}$ is usually not problematic because the number of factors $r$ is sufficiently small compared to the sample size. However, it is theoretically convenient to (formally) define $\hat{\beta}_n$ in the case that $\hat{\Sigma}_{X,n}$ is singular. For this reason we take an $\mcl{S}_d^{++}$-valued random variable $\hat{\Sigma}_{X,n}^{\dagger}$ such that $\hat{\Sigma}_{X,n}^{\dagger}=\hat{\Sigma}_{X,n}^{-1}$ on the event where $\hat{\Sigma}_{X,n}$ is invertible, and redefine $\hat{\beta}_n$ as $\hat{\beta}_n:=\hat{\Sigma}_{YX,n}\hat{\Sigma}_{X,n}^{\dagger}$. This does not affect the asymptotic properties of our estimators because $\hat{\Sigma}_{X,n}$ is asymptotically invertible under our assumptions we will impose. 
Now, from \eqref{eq:factor-qc}, $\Sigma_Z$ is estimated by
\begin{equation}\label{def:sigma.z.hat}
\hat{\Sigma}_{Z,n}:=\hat{\Sigma}_{Y,n}-\hat{\beta}_n\hat{\Sigma}_{X,n}\hat{\beta}_n^\top.
\end{equation}
Since $\hat{\Sigma}_{Z,n}$ might be a poor estimator for $\Sigma_Z$ because $d$ can be extremely large in our setting, we apply the weighted graphical Lasso to $\hat{\Sigma}_{Z,n}$ in order to estimate $\Sigma_Z$. Namely, we construct the weighted graphical Lasso estimator $\hat{\Theta}_{Z,\lambda}$ based on $\hat{\Sigma}_{Z,n}$ as follows:
\begin{equation}\label{f-wglasso}
\hat{\Theta}_{Z,\lambda}=\arg\min_{\Theta\in\mcl{S}_d^{++}}\left\{\trace\left(\Theta \hat{\Sigma}_{Z,n}\right)-\log\det\left(\Theta\right)+\lambda\sum_{i\neq j}\sqrt{\hat{\Sigma}_{Z,n}^{ii}\hat{\Sigma}_{Z,n}^{jj}}\left|\Theta^{ij}\right|\right\}.
\end{equation}
Then $\Sigma_Z$ is estimated by the inverse of $\hat{\Theta}_{Z,\lambda}$. Hence our final estimator for $\Sigma_Y$ is constructed as
\begin{equation}\label{def:f-sigma.y}
\hat{\Sigma}_{Y,\lambda}:=\hat{\beta}_n\hat{\Sigma}_{X,n}\hat{\beta}_n^\top+\hat{\Theta}_{Z,\lambda}^{-1}.
\end{equation}
\begin{rmk}
Although we will impose the assumptions which guarantee that the optimization problem in \eqref{f-wglasso} asymptotically has the unique solution with probability 1, it may have no solution for a fixed $n$. Thus, we formally define $\hat{\Theta}_{Z,\lambda}$ as an $\mcl{S}_d^{++}$-valued random variable such that $\hat{\Theta}_{Z,\lambda}$ is defined by \eqref{f-wglasso} on the event where the optimization problem in \eqref{f-wglasso} has the unique solution. 
\end{rmk}
\begin{rmk}[Positive definiteness of $\hat{\Sigma}_{Y,\lambda}$]
Since $\hat{\Theta}_{Z,\lambda}^{-1}$ is positive definite by construction, $\hat{\Sigma}_{Y,\lambda}$ is positive definite (note that we assume $\hat{\Sigma}_{X,n}$ is positive semidefinite). 
\end{rmk}

We will impose the following structural assumptions on the model:
\begin{enumerate}[label={\normalfont[C\arabic*]}]

\item\label{ass:vol} $\|\Sigma_Y\|_{\ell_\infty}=O_p(1)$ and $\|\beta\|_{\ell_\infty}=O(1)$ as $n\to\infty$.

\item\label{ass:eigen.z} $\Lambda_\mathrm{max}(\Sigma_Z)+1/\Lambda_\mathrm{min}(\Sigma_Z)=O_p(1)$ as $n\to\infty$.

\item\label{ass:eigen.f} $\|\Sigma_X\|_{\ell_\infty}+1/\Lambda_\mathrm{min}(\Sigma_X)=O_p(1)$ as $n\to\infty$. 

\item\label{ass:sparse.z} $s(\Theta_Z)=O_p(s_n)$ as $n\to\infty$ for some sequence $s_n\in[1,\infty)$, $n=1,2,\dots$. 

\item\label{ass:degree.z} $\mf{d}(\Theta_Z)=O_p(\mf{d}_n)$ as $n\to\infty$ for some sequence $\mf{d}_n\in[1,\infty)$, $n=1,2,\dots$. 

\item\label{ass:loading} There is a positive definite $d\times d$ matrix $\bs{B}$ such that $\opnorm{d^{-1}\beta^\top\beta-\bs{B}}_2\to0$ and $\Lambda_{\min}(\bs{B})^{-1}=O(1)$ as $n\to\infty$.

\end{enumerate}

\ref{ass:vol}--\ref{ass:eigen.f} are natural structural assumptions on the model and standard in the literature; see e.g.~Assumptions 2.1 and 3.3 in \cite{FLM2011}. \ref{ass:sparse.z}--\ref{ass:degree.z} are sparsity assumptions on the precision matrix of the residual process and necessary for our application of the (weighted) graphical Lasso. \ref{ass:loading} requires the factors to have non-negligible impact on almost all assets and is also standard in the context of covariance matrix estimation based on a factor model; see e.g.~Assumption 3.5 in \cite{FLM2011} and Assumption 6 in \cite{FFX2016}.

The following result establishes the consistency of the residual precision matrix estimator $\hat{\Theta}_{Z,\lambda}$. 
\begin{proposition}\label{factor:rate}
Assume \ref{ass:vol}--\ref{ass:sparse.z}. Let $(\lambda_n)_{n=1}^\infty$ be a sequence of positive-valued random variables satisfying the following conditions:
\begin{enumerate}[label={\normalfont[D\arabic*]}]

\item\label{f-ass:est} $\lambda_n^{-1}\|\hat{\Sigma}_{X,n}-\Sigma_{X}\|_{\ell_\infty}\to^p0$, $\lambda_n^{-1}\|\hat{\Sigma}_{YX,n}-\beta\hat{\Sigma}_{X,n}\|_{\ell_\infty}\to^p0$ and $\lambda_n^{-1}\|\breve{\Sigma}_{Z,n}-\Sigma_Z\|_{\ell_\infty}\to^p0$ as $n\to\infty$, where $\breve{\Sigma}_{Z,n}:=\hat{\Sigma}_{Y,n}-\hat{\Sigma}_{YX,n}\beta^\top-\beta\hat{\Sigma}_{YX,n}^\top+\beta\hat{\Sigma}_{X,n}\beta^\top$.

\item\label{f-ass:rate} $(s_n+r)\lambda_n\to^p0$ as $n\to\infty$.

\item\label{f-ass:psd} $P(\ol{\Sigma}_n\in\mcl{S}_d^+)\to1$ as $n\to\infty$, where
\[
\ol{\Sigma}_n:=
\begin{pmatrix}
\hat{\Sigma}_{X,n} & \hat{\Sigma}_{YX,n}^\top\\
\hat{\Sigma}_{YX,n} & \hat{\Sigma}_{Y,n}
\end{pmatrix}.
\]

\end{enumerate}
Then $\lambda_n^{-1}\opnorm{\hat{\Theta}_{Z,\lambda_n}-\Theta_Z}_w=O_p(s_n)$ and $\lambda_n^{-1}\opnorm{\hat{\Theta}_{Z,\lambda_n}^{-1}-\Sigma_Z}_2=O_p(s_n)$ as $n\to\infty$ for any $w\in[1,\infty]$. 
\end{proposition}

\begin{rmk}
(a) Since $\Sigma_{ZX}=\Sigma_{YX}-\beta\Sigma_X$ and $\Sigma_{Z}=\Sigma_{Y}-\Sigma_{YX}\beta^\top-\beta\Sigma_{XY}+\beta\Sigma_X\beta^\top$, $\hat{\Sigma}_{YX,n}$ and $\breve{\Sigma}_{Z,n}$ are seen as natural estimators for $\Sigma_{ZX}(=0)$ and $\Sigma_Z$ respectively if $\beta$ were known. In this sense, \ref{f-ass:est} is a natural extension of \ref{ass:est}. In particular, if $r=O(1)$ as $n\to\infty$, \ref{f-ass:est} follows from the convergences $\lambda_n^{-1}\|\hat{\Sigma}_{X,n}-\Sigma_{X}\|_{\ell_\infty}\to^p0$, $\lambda_n^{-1}\|\hat{\Sigma}_{YX,n}-\Sigma_{YX}\|_{\ell_\infty}\to^p0$ and $\lambda_n^{-1}\|\hat{\Sigma}_{Y,n}-\Sigma_Y\|_{\ell_\infty}\to^p0$ under \ref{ass:vol}, which are typically derived from entry-wise concentration inequalities for $\hat{\Sigma}_{X,n},\hat{\Sigma}_{YX,n}$ and $\hat{\Sigma}_{Y,n}$.  

(b) \ref{f-ass:psd} ensures that $\hat{\Sigma}_{Z,n}$ is asymptotically positive semidefinite. This is necessary for guaranteeing that the optimization problem in \eqref{f-wglasso} asymptotically has the unique solution with probability 1.
\end{rmk}

From Proposition \ref{factor:rate} we can also derive the convergence rates for the estimators $\hat{\Sigma}_{Z,\lambda_n}$ and $\hat{\Sigma}_{Z,\lambda_n}^{-1}$ in appropriate norms, which may be seen as counterparts of Theorems 1--2 in \cite{FFX2016}. 
\begin{proposition}\label{factor:max-rate}
Under the assumptions of Proposition \ref{factor:rate}, $\lambda_n^{-1}\|\hat{\Sigma}_{Z,\lambda_n}-\Sigma_Z\|_{\ell_\infty}=O_p(s_n+r^2)$ as $n\to\infty$
\end{proposition}

\begin{proposition}\label{factor:inverse}
Under the assumptions of Proposition \ref{factor:rate}, we additionally assume \ref{ass:degree.z}--\ref{ass:loading}. Then, $\lambda_n^{-1}\opnorm{\hat{\Sigma}_{Y,\lambda_n}^{-1}-\Sigma_Y^{-1}}_2=O_p(s_n+r)$ and $\lambda_n^{-1}\opnorm{\hat{\Sigma}_{Y,\lambda_n}^{-1}-\Sigma_Y^{-1}}_\infty=O_p(r^{3/2}\mf{d}_n(s_n+r))$ as $n\to\infty$.
\end{proposition}

Next we present the high-dimensional asymptotic mixed normality of the de-biased version of $\hat{\Theta}_{Z,\lambda}$. 
\begin{proposition}\label{factor:AMN}
Suppose that the assumptions of Proposition \ref{factor:rate} and \ref{ass:degree.z} are satisfied. 
For every $n\in\mathbb{N}$, let $a_n>0$, $\mathfrak{C}_n$ be a $d^2\times d^2$ positive semidefinite random matrix and $J_n$ be an $m\times d^2$ random matrix, where $m=m_n$ may depend on $n$. 
Assume $a_n\opnorm{J_n}_\infty\lambda_n^2s_n\sqrt{\mf{d}_n\log (m+1)}\to^p0$ as $n\to\infty$. 
Assume also that
\begin{equation}\label{f-est:AMN}
\lim_{n\to\infty}\sup_{y\in\mathbb{R}^m}\left|P\left(a_n\wt{J}_{Z,n}\vectorize\left(\breve{\Sigma}_{Z,n}-\Sigma_Z\right)\leq y\right)-P\left(\wt{J}_{Z,n}\mathfrak{C}_n^{1/2}\zeta_n\leq y\right)\right|=0
\end{equation}
and
\begin{equation}\label{factor:diag-tight}
\lim_{b\downarrow0}\limsup_{n\to\infty}P(\min\diag(\wt{J}_{Z,n}\mathfrak{C}_n\wt{J}_{Z,n}^\top)<b)=0
\end{equation}
as $n\to\infty$, where $\wt{J}_{Z,n}:=-J_n(\Theta_Z\otimes\Theta_Z)$ and $\zeta_n$ is a $d^2$-dimensional standard Gaussian vector independent of $\mathcal{F}$, which is defined on an extension of the probability space $(\Omega,\mathcal{F},P)$ if necessary. 
Then,
\[
\lim_{n\to\infty}\sup_{y\in\mathbb{R}^m}\left|P\left(a_nJ_n\vectorize\left(\hat{\Theta}_{Z,\lambda_n}-\Gamma_{Z,n}-\Theta_Z\right)\leq y\right)-P\left(\wt{J}_{Z,n}\mathfrak{C}_n^{1/2}\zeta_n\leq y\right)\right|=0,
\]
where $\Gamma_{Z,n}:=-(\hat{\Theta}_{Z,\lambda_n}-\hat{\Theta}_{Z,\lambda_n}\hat{\Sigma}_{Z,n}\hat{\Theta}_{Z,\lambda_n})$.  
\end{proposition}

\begin{rmk}
It is worth mentioning that condition \eqref{f-est:AMN} is stated for $\breve{\Sigma}_{Z,n}$ rather than $\hat{\Sigma}_{Z,n}$. In other words, for deriving the asymptotic distribution, we do not need to take account of the effect of plugging $\hat{\beta}_n$ into $\beta$, at least in the first order. This is thanks to Lemma \ref{lemma:sigma.z}. 
\end{rmk}

Although it is generally difficult to derive the asymptotic mixed normality of (the de-biased version of) $\hat{\Sigma}_{Y,\lambda_n}^{-1}$, this is possible when $d$ is sufficiently large. In fact, in such a situation, the entry-wise behavior of $\Sigma_Y^{-1}$ is dominated by $\Theta_Z$ as described by the following lemma:
\begin{lemma}\label{lemma:f-inv-AMN}
Under the assumptions of Proposition \ref{factor:inverse}, $\|\hat{\Sigma}_{Y,\lambda_n}^{-1}-\hat{\Theta}_{Z,\lambda_n}\|_{\ell_\infty}=O_p(r\mf{d}_n/d)$ and $\|\Sigma_{Y}^{-1}-\Theta_{Z}\|_{\ell_\infty}=O_p(r\mf{d}_n/d)$ as $n\to\infty$. 
\end{lemma}

As a consequence, we obtain the following result. 
\begin{proposition}\label{coro:f-inv-AMN}
Suppose that the assumptions of Proposition \ref{factor:AMN} and \ref{ass:loading} are satisfied. Suppose also $a_n\opnorm{J_n}_\infty r\mf{d}_n\sqrt{\log (m+1)}/d\to0$ as $n\to\infty$. Then we have
\[
\lim_{n\to\infty}\sup_{y\in\mathbb{R}^m}\left|P\left(a_nJ_n\vectorize\left(\hat{\Sigma}_{Y,\lambda_n}^{-1}-\Gamma_{Z,n}-\Sigma_Y^{-1}\right)\leq y\right)-P\left(\wt{J}_{Z,n}\mathfrak{C}_n^{1/2}\zeta_n\leq y\right)\right|=0.
\]
\end{proposition}

\section{Application to realized covariance matrix}\label{sec:rc}

In this section we apply the abstract theory developed above to the simplest situation where the processes have no jumps and are observed at equidistant times without noise. 
%
Specifically, we consider the continuous-time factor model \eqref{factor-model} and assume that both $Y$ and $X$ are observed at equidistant time points $h/n$, $h=0,1,\dots,n$. 
In this case, $\Sigma_Y=[Y,Y]_1$ is naturally estimated by the \textit{realized covariance matrix}:
\begin{equation}\label{def:rc}
\hat{\Sigma}_{Y,n}:=\wh{[Y,Y]}_1^n:=\sum_{h=1}^{n}(Y_{h/n}-Y_{(h-1)/n})(Y_{h/n}-Y_{(h-1)/n})^\top.
\end{equation}
Analogously, we define $\hat{\Sigma}_{X,n}:=\wh{[X,X]}_1^n$ and $\hat{\Sigma}_{YX,n}:=\wh{[Y,X]}_1^n$. 
In addition, we assume that $Z$ and $X$ are respectively $d$-dimensional and $r$-dimensional continuous It\^o semimartingales given by 
\begin{equation*}
Z_t=Z_0+\int_0^t\mu_sds+\int_0^t\sigma_sdW_s,\qquad
X_t=X_0+\int_0^t\wt{\mu}_sds+\int_0^t\wt{\sigma}_sdW_s,
\end{equation*}
where $\mu=(\mu_s)_{s\in[0,1]}$ and $\wt{\mu}=(\wt{\mu}_s)_{s\in[0,1]}$ are respectively $d$-dimensional and $r$-dimensional $(\mathcal{F}_t)$-progressively measurable processes, $\sigma=(\sigma_s)_{s\in[0,1]}$ and $\wt{\sigma}=(\wt{\sigma}_s)_{s\in[0,1]}$ are respectively $\mathbb{R}^{d\times d'}$-valued and $\mathbb{R}^{r\times d'}$-valued $(\mathcal{F}_t)$-progressively measurable processes, and $W=(W_s)_{s\in[0,1]}$ is a $d'$-dimensional standard $(\mathcal{F}_t)$-Wiener process. 
To apply the convergence rate results to this setting, we impose the following assumptions:
\begin{enumerate}[label={\normalfont[E\arabic*]}]

\item\label{rc-bdd} For all $n,\nu\in\mathbb{N}$, we have an event $\Omega_n(\nu)\in\mathcal{F}$ and $(\mathcal{F}_t)$-progressively measurable processes $\mu(\nu)=(\mu(\nu)_s)_{s\in[0,1]}$, $\wt{\mu}(\nu)=(\wt{\mu}(\nu)_s)_{s\in[0,1]}$, $\sigma(\nu)=(\sigma(\nu)_s)_{s\in[0,1]}$ and $\wt{\sigma}(\nu)=(\wt{\sigma}(\nu)_s)_{s\in[0,1]}$ which take values in $\mathbb{R}^d$, $\mathbb{R}^r$, $\mathbb{R}^{d\times d'}$ and $\mathbb{R}^{r\times d'}$, respectively, and they satisfy the following conditions: 
\begin{enumerate}[label={\normalfont(\roman*)}]

\item $\lim_{\nu\to\infty}\limsup_{n\to\infty}P(\Omega_n(\nu)^c)=0$.

\item $\mu=\mu(\nu)$, $\wt{\mu}=\wt{\mu}(\nu)$, $\sigma=\sigma(\nu)$ and $\wt{\sigma}=\wt{\sigma}(\nu)$ on $\Omega_n(\nu)$ for all $\nu\in\mathbb{N}$.  

\item For all $\nu\in\mathbb{N}$, there is a constant $C_\nu>0$ such that
\[
\sup_{n\in\mathbb{N}}\sup_{0\leq t\leq1}\sup_{\omega\in\Omega}\left(\|\mu(\nu)_t(\omega)\|_{\ell_\infty}+\|\wt{\mu}(\nu)_t(\omega)\|_{\ell_\infty}+\|c(\nu)_t(\omega)\|_{\ell_\infty}
+\|\tilde{c}(\nu)_t(\omega)\|_{\ell_\infty}\right)\leq C_\nu,
\]
where $c(\nu)_t:=\sigma(\nu)_t\sigma(\nu)_t^\top$ and $\wt{c}(\nu)_t:=\wt{\sigma}(\nu)_t\wt{\sigma}(\nu)_t^\top$. 

\end{enumerate}

\item\label{dim-rate} $r=O(d)$ and $(\log d)/\sqrt{n}\to0$ as $n\to\infty$. 

\end{enumerate}

\ref{rc-bdd} is a local boundedness assumption on the coefficient processes and typical in the literature: For example, \ref{rc-bdd} is satisfied when $\mu,\wt{\mu},\sigma$ and $\wt{\sigma}$ are all bounded by some locally bounded process independent of $n$. This latter condition is imposed in \cite{FFX2016}, among others. 
\ref{dim-rate} restricts the growth rates of $d$ and $r$. It is indeed an adaptation of \ref{f-ass:est} to the present setting. 

\begin{theorem}\label{rc:rate}
Assume \ref{ass:vol}--\ref{ass:sparse.z} and \ref{rc-bdd}--\ref{dim-rate}. 
Let $\lambda_n$ be a sequence of positive-valued random variables such that $\lambda_n^{-1}\sqrt{(\log d)/n}\to^p0$ and $(s_n+r)\lambda_n\to^p0$ as $n\to\infty$. Then $\lambda_n^{-1}\opnorm{\hat{\Theta}_{Z,\lambda_n}-\Theta_Z}_w=O_p(s_n)$, $\lambda_n^{-1}\opnorm{\hat{\Theta}_{Z,\lambda_n}^{-1}-\Sigma_Z}_2=O_p(s_n)$ and $\lambda_n^{-1}\|\hat{\Sigma}_{Y,\lambda_n}-\Sigma_Y\|_{\ell_\infty}=O_p(s_n+r^2)$ as $n\to\infty$ for any $w\in[1,\infty]$. 
Moreover, if we additionally assume \ref{ass:degree.z}--\ref{ass:loading}, then $\lambda_n^{-1}\opnorm{\hat{\Sigma}_{Y,\lambda_n}^{-1}-\Sigma_Y^{-1}}_2=O_p(s_n+r)$ and $\lambda_n^{-1}\opnorm{\hat{\Sigma}_{Y,\lambda_n}^{-1}-\Sigma_Y^{-1}}_\infty=O_p(r^{3/2}\mf{d}_n(s_n+r))$ as $n\to\infty$.
\end{theorem}

\begin{rmk}[Optimal convergence rate]
From Theorem \ref{rc:rate}, the convergence rate of $\hat{\Theta}_{Z,\lambda_n}$ to $\Theta_Z$ in the $\ell_w$-operator norm for any $w\in[1,\infty]$ can be arbitrarily close to $s_n\sqrt{(\log d)/n}$, which is similar to that in a standard i.i.d.~setting (cf.~Theorem 14.1.3 in \cite{JvdG2018}). 
On the other hand, in the Gaussian i.i.d.~setting without factor structure, the minimax optimal rate for this problem is known to be $\mf{d}(\Theta_Z)\sqrt{(\log d)/n}$ (see \cite[Theorem 1.1]{CLZ2016} and \cite[Theorem 5]{CRZ2016}), which can be faster than $s_n\sqrt{(\log d)/n}$. In a standard i.i.d.~setting, this rate can be attained by using a node-wise penalized regression (see e.g.~\cite[Section 3.1]{CRZ2016}), so it would be interesting to study the convergence rate of such a method in our setting. We leave it to future research. 
In the meantime, such a method does not ensure the positive definiteness of the estimated precision matrix in general, so our estimator would be preferable for some practical applications such as portfolio allocation. 
\end{rmk}

Next we derive the asymptotic mixed normality of the de-biased estimator in the present setting. 
As announced, we accomplish this purpose with the help of Malliavin calculus. In the following we will freely use standard concepts and notation from Malliavin calculus. We refer to \cite{Nualart2006} and \cite[Chapter 15]{Janson1997} for detailed treatments of this subject. 

We consider the Malliavin calculus with respect to $W$. 
For any real number $p\geq1$ and any integer $k\geq1$, $\mathbb{D}_{k,p}$ denotes the stochastic Sobolev space of random variables which are $k$ times differentiable in the Malliavin sense and the derivatives up to order $k$ have finite moments of order $p$. 
If $F\in\mathbb{D}_{k,p}$, we denote by $D^kF$ the $k$th Malliavin derivative of $F$, which is a random variable taking values in $L^2([0,1]^k;(\mathbb{R}^{d'})^{\otimes k})$. 
Here, we identify the space $(\mathbb{R}^{d'})^{\otimes k}$ with the set of all $d'$-dimensional $k$-way arrays, i.e.~real-valued functions on $\{1,\dots,d'\}^k$. 
Since $D^kF$ is a random function on $[0,1]^k$, we can consider the value $D^kF(t_1,\dots,t_k)$ evaluated at $(t_1,\dots,t_k)\in[0,1]^k$. We denote this value by $D_{t_1,\dots,t_k}F$. Moreover, since $D_{t_1,\dots,t_k}F$ takes values in $(\mathbb{R}^{d'})^{\otimes k}$, we can consider the value $D_{t_1,\dots,t_k}F(a_1,\dots,a_k)$ evaluated at $(a_1,\dots,a_k)\in\{1,\dots,d'\}^k$. This value is denoted by $D^{(a_1,\dots,a_k)}_{t_1,\dots,t_k}F$. 
We remark that the variable $D_{t_1,\dots,t_k}F$ is defined only a.e.~on $[0,1]^k\times\Omega$ with respect to the measure $\leb_k\times P$, where $\leb_k$ denotes the Lebesgue measure on $[0,1]^k$. Therefore, if $D_{t_1,\dots,t_k}F$ satisfies some property a.e.~on $[0,1]^k\times\Omega$ with respect to $\leb_k\times P$, by convention we will always take a version of $D_{t_1,\dots,t_k}F$ satisfying that property everywhere on $[0,1]^k\times\Omega$ if necessary. 
We set $\mathbb{D}_{k,\infty}:=\bigcap_{p=1}^\infty\mathbb{D}_{k,p}$. 
We denote by $\mathbb{D}_{k,\infty}(\mathbb{R}^d)$ the space of all $d$-dimensional random variables $F$ such that $F^i\in\mathbb{D}_{k,\infty}$ for every $i=1,\dots,d$. The space $\mathbb{D}_{k,\infty}(\mathbb{R}^{d\times r})$ is defined in an analogous way. 
Finally, for any $(\mathbb{R}^{d'})^{\otimes k}$-valued random variable $F$ and $p\in(0,\infty]$, we set
\[
\|F\|_{p,\ell_2}:=\left\|\sqrt{\sum_{a_1,\dots,a_k=1}^{d'}F(a_1,\dots,a_k)^2}\right\|_p.
\]

We also need to define some variables related to the ``asymptotic'' covariance matrices of the estimators. 
We define $d^2\times d^2$ random matrix $\mathfrak{C}_n$ by
\begin{multline*}
\mathfrak{C}_n^{(i-1)d+j,(k-1)d+l}:=\\
n\sum_{h=1}^n\left\{\left(\int_{(h-1)/n}^{h/n}c_s^{ik}ds\right)\left(\int_{(h-1)/n}^{h/n}c_s^{jl}ds\right)
+\left(\int_{(h-1)/n}^{h/n}c_s^{il}ds\right)\left(\int_{(h-1)/n}^{h/n}c_s^{jk}ds\right)\right\},\\
i,j,k,l=1,\dots,d,
\end{multline*}
where $c_s:=\sigma_s\sigma_s^\top$. 
Then we set $\mf{V}_n:=(\Theta_{Z}\otimes\Theta_{Z})\mathfrak{C}_n(\Theta_{Z}\otimes\Theta_{Z})$ and $\mf{S}_n:=\diag(\mf{V}_n)^{1/2}$. 
In addition, under \ref{rc-bdd}, we define $\mf{C}_n(\nu)$ similarly to $\mf{C}_n$ with replacing $\sigma$ by $\sigma(\nu)$. 
$\mf{C}_n$ and $\mf{V}_n$ play roles of the asymptotic covariance matrices of $\breve{\Sigma}_{Z,n}$ and $\hat{\Theta}_{Z,\lambda_n}$, respectively. 

We impose the following assumptions on the model. 
\begin{enumerate}[label={\normalfont[F\arabic*]}]

\item\label{rc-mal} We have \ref{rc-bdd} and $\Sigma_Z(\nu):=\int_0^1c(\nu)_tdt$ is a.s.~invertible for all $n,\nu\in\mathbb{N}$. Moreover, for all $n,\nu\in\mathbb{N}$ and $t\in[0,1]$, $\mu(\nu)_t\in\mathbb{D}_{1,\infty}(\mathbb{R}^{d})$, $\sigma(\nu)_t\in\mathbb{D}_{2,\infty}(\mathbb{R}^{d\times r})$ and
\begin{align}
&\sup_{n\in\mathbb{N}}
\max_{1\leq i\leq d}
\sup_{0\leq s,t\leq 1}\|D_s\mu(\nu)_t^{i}\|_{\infty,\ell_2}
<\infty,\label{eq-mu}\\
&\sup_{n\in\mathbb{N}}
\max_{1\leq i\leq d}\left(
\sup_{0\leq s,t\leq 1}\|D_s\sigma(\nu)_t^{i\cdot}\|_{\infty,\ell_2}
+\sup_{0\leq s,t,u\leq 1}\|D_{s,t}\sigma(\nu)_u^{i\cdot}\|_{\infty,\ell_2}
\right)
<\infty,\label{eq-sigma}\\
&\sup_{n\in\mathbb{N}}\lpa\max_{1\leq i\leq d}\|\Theta_Z(\nu)^{ii}\|_\infty+\max_{1\leq k\leq d^2}\|1/\mathfrak{V}_n(\nu)^{kk}\|_\infty\rpa<\infty,\label{Theta-bdd}
\end{align}
where $\Theta_Z(\nu):=\Sigma_Z(\nu)^{-1}$ and $\mf{V}_n(\nu):=(\Theta_{Z}(\nu)\otimes\Theta_{Z}(\nu))\mathfrak{C}_n(\nu)(\Theta_{Z}(\nu)\otimes\Theta_{Z}(\nu))$. 

\item\label{ass:adjacent} The $d\times d$ matrix $Q_{Z}:=(1_{\{\Theta_Z^{ij}\neq0\}})_{1\leq i,j\leq d}$ is non-random and $\mf{d}(Q_Z)=O(1)$ as $n\to\infty$. 

\item\label{dim-AMN} $r=O(d)$ and $(\log d)^{13}/n\to0$ as $n\to\infty$. 

\end{enumerate}

We give a few remarks on these assumptions. 
First, \ref{rc-mal} imposes the (local) Malliavin differentiability on the coefficient processes of the residual process $Z$ and the local boundedness on their Malliavin derivatives. Such an assumption is necessary for the application of the high-dimensional mixed normal limit theorem of \cite{Koike2018sk} to our setting (see Lemma \ref{lemma:rc-AMN}). 
Note that we do not need to impose this type of assumption on the factor process $X$. 
We also remark that analogous assumptions are sometimes used in the literature of high-frequency financial econometrics even in low-dimensional settings; see e.g.~\cite{CG2011,CPTV2017}. 
Second, \ref{ass:adjacent} is clearly understood when we consider a Gaussian graphical model associated with $\Sigma_Z$: The non-randomness of $Q_Z$ implies that the edge structure of this Gaussian graphical model is determined in a non-random manner.\footnote{By conditioning, it is indeed sufficient that the edge structure is determined independently of the driving Wiener process $W$.} 
Also, we remark that the condition $\mf{d}(Q_Z)=O(1)$ is equivalent to \ref{ass:degree.z} with $\mf{d}_n=1$. It is seemingly possible to relax this condition so that it allows a diverging sequence $\mf{d}_n$ as long as $\mf{d}_n(\log d)^\kappa/n\to0$ for an appropriate constant $\kappa>0$. However, to determine the precise value of $\kappa$, we need to carefully revise the proof of Lemma \ref{lemma:rc-AMN} so that it allows the quantity inside $\sup_{n\in\mathbb{N}}$ in \eqref{eq-X} to diverge as $n\to\infty$. To avoid such an additional complexity, we restrict our attention to the case of $\mf{d}_n=1$. 
Third, the condition $(\log d)^{13}/n\to0$ in \ref{dim-AMN} is used again for applying the high-dimensional CLT of \cite{Koike2018sk}. 

Now we are ready to state our result. 
Let $\mathcal{A}^\mathrm{re}(d^2)$ be the set of all hyperrectangles in $\mathbb{R}^{d^2}$, i.e.~$\mathcal{A}^\mathrm{re}(d^2)$ consists of all sets $A$ of the form 
$
A=\{x\in\mathbb{R}^{d^2}:a_j\leq x^j\leq b_j\text{ for all }j=1,\dots,d^2\}
$ 
for some $-\infty\leq a_j\leq b_j\leq\infty$, $j=1,\dots,d^2$. 
\begin{theorem}\label{rc:AMN}
Assume \ref{ass:vol}--\ref{ass:sparse.z} and \ref{rc-mal}--\ref{dim-AMN}.  
Let $\lambda_n$ be a sequence of positive-valued random variables such that $\lambda_n^{-1}\sqrt{(\log d)/n}\to^p0$, $(s_n+r)\lambda_n\to^p0$ and $\lambda_n^2s_n\sqrt{n\log d}\to^p0$ as $n\to\infty$. 
Then we have
\begin{equation}\label{AMN:non-student}
\sup_{A\in\mathcal{A}^\mathrm{re}(d^2)}\left|P\left(\sqrt{n}\vectorize(\hat{\Theta}_{Z,\lambda_n}-\Gamma_{Z,n}-\Theta_Z)\in A\right)-P\left(\mf{V}_n^{1/2}\zeta_n\in A\right)\right|\to0
\end{equation}
and
\begin{equation}\label{AMN:student}
\sup_{A\in\mathcal{A}^\mathrm{re}(d^2)}\left|P\left(\sqrt{n}\mf{S}_n^{-1}\vectorize(\hat{\Theta}_{Z,\lambda_n}-\Gamma_{Z,n}-\Theta_Z)\in A\right)-P\left(\mf{S}_n^{-1}\mf{V}_n^{1/2}\zeta_n\in A\right)\right|\to0
\end{equation}
as $n\to\infty$.
\end{theorem}

\begin{rmk}
$\lambda_n$ is typically chosen of order close to $\sqrt{\log d/n}$ as possible, so $\lambda_n^2s_n\sqrt{n\log d}\to^p0$ is almost equivalent to $s_n(\log d)^{\frac{3}{2}}/\sqrt{n}\to0$. This is stronger than the condition $s_n(\log d)/\sqrt{n}\to0$ which is used to derive the asymptotic normality of the de-biased weighted graphical Lasso estimator in \cite[Theorem 14.1.6]{JvdG2018} (note that we assume $\mf{d}(\Theta_Z)=O_p(1)$). This is because Theorem \ref{rc:AMN} derives the approximations of the joint distributions of the de-biased estimator and its Studentization, while \cite[Theorem 14.1.6]{JvdG2018} focuses only on approximation of their marginal distributions.  
\end{rmk}

Theorem \ref{rc:AMN} is statistically infeasible in the sense that $\mf{V}_n$ is unobservable. Thus, we need to estimate it from the data. Since $\Theta_Z$ is naturally estimated by $\hat{\Theta}_{Z,\lambda_n}$, we construct an estimator for $\mf{C}_n$. 
Define the $d^2$-dimensional random vectors $\hat{\chi}_h$ by
\[
\hat{\chi}_{h}:=\vectorize\left[(\hat{Z}_{h/n}-\hat{Z}_{(h-1)/n})(\hat{Z}_{h/n}-\hat{Z}_{(h-1)/n})^\top\right],\qquad
h=1,\dots,n,
\]
where $\hat{Z}_{h/n}:=Y_{h/n}-\hat{\beta}X_{h/n}$. 
Then we set
\[
\hat{\mathfrak{C}}_n:=n\sum_{h=1}^n\hat{\chi}_h\hat{\chi}_h^\top-\frac{n}{2}\sum_{h=1}^{n-1}\left(\hat{\chi}_h\hat{\chi}_{h+1}^\top+\hat{\chi}_{h+1}\hat{\chi}_{h}^\top\right).
\]
\begin{lemma}\label{lemma:avar}
Suppose that the assumptions of Theorem \ref{rc:AMN} are satisfied. 
Suppose also $r^2(\log d)/n=O(1)$ as $n\to\infty$ and that there is a constant $\gamma\in(0,\frac{1}{2}]$ such that
\begin{equation}\label{sigma-modulus}
\sup_{0< t\leq1-\frac{1}{n}}\left\|\max_{1\leq i,j\leq d}\left|c(\nu)^{ij}_{t+\frac{1}{n}}-c(\nu)^{ij}_t\right|\right\|_2=O(n^{-\gamma})
\end{equation}
as $n\to\infty$ for all $\nu\in\mathbb{N}$. 
Then, $\|\hat{\mf{C}}_n-\mf{C}_n\|_{\ell_\infty}=O_p(r(\log d)^{5/2}/\sqrt{n}+n^{-\gamma})$ as $n\to\infty$. 
\end{lemma}

Let us set $\hat{\mf{V}}_n:=(\hat{\Theta}_{Z,\lambda_n}\otimes\hat{\Theta}_{Z,\lambda_n})\hat{\mathfrak{C}}_n(\hat{\Theta}_{Z,\lambda_n}\otimes\hat{\Theta}_{Z,\lambda_n})$ and $\hat{\mf{S}}_n:=\diag(\hat{\mf{V}}_n)$. 
\begin{corollary}\label{coro:feasible}
Under the assumptions of Lemma \ref{lemma:avar}, we have the following results:
\begin{enumerate}[label=(\alph*)]

\item\label{feasible-a} 
Assume $s_n\lambda_n\log d\to^p0$ and $r(\log d)^{7/2}/\sqrt{n}+n^{-\gamma}\log d\to0$ as $n\to\infty$. Then,
\[
\lim_{n\to\infty}\sup_{A\in\mathcal{A}^\mathrm{re}(d^2)}\left|P\left(\sqrt{n}\hat{\mf{S}}_n^{-1}\vectorize(\hat{\Theta}_{Z,\lambda_n}-\Gamma_{Z,n}-\Theta_Z)\in A\right)-P\left(\mf{S}_n^{-1}\mf{V}_n^{1/2}\zeta_n\in A\right)\right|=0.
\]

\item\label{feasible-b} 
Assume $s_n\lambda_n(\log d)^2\to^p0$ and $r(\log d)^{9/2}/\sqrt{n}+n^{-\gamma}(\log d)^2\to0$ as $n\to\infty$. Then,
\begin{align*}
&\sup_{A\in\mathcal{A}^\mathrm{re}(d^2)}\left|P\left(\hat{\mf{V}}_n^{1/2}\zeta_n\in A\mid\mcl{F}\right)-P\left(\mf{V}_n^{1/2}\zeta_n\in A\mid\mcl{F}\right)\right|\to^p0,\\
&\sup_{A\in\mathcal{A}^\mathrm{re}(d^2)}\left|P\left(\hat{\mf{S}}_n^{-1}\hat{\mf{V}}_n^{1/2}\zeta_n\in A\mid\mcl{F}\right)-P\left(\mf{S}_n^{-1}\mf{V}_n^{1/2}\zeta_n\in A\mid\mcl{F}\right)\right|\to^p0
\end{align*}
as $n\to\infty$.

\end{enumerate}
\end{corollary}
Corollary \ref{coro:feasible}\ref{feasible-a} particularly implies that
\begin{equation}\label{eq:point-AN}
\lim_{n\to\infty}\max_{1\leq i,j\leq d}\sup_{x\in\mathbb{R}}\labs P\lpa\frac{\sqrt{n}\lpa\hat{\Theta}_{Z,\lambda_n}^{ij}-\Gamma_{Z,n}^{ij}-\Theta_Z^{ij}\rpa}{\hat{\mf{S}}_n^{(i-1)d+j,(i-1)d+j}}\leq x\rpa-\Phi(x)\rabs=0,
\end{equation}
where $\Phi$ is the standard normal distribution function. This result can be used to construct entry-wise confidence intervals for $\Theta_Z$. 
Meanwhile, combining Corollary \ref{coro:feasible}\ref{feasible-b} with \cite[Proposition 3.2]{Koike2018sk}, we can estimate the quantiles of $\max_{k\in\mcl{K}}(\mf{V}_n^{1/2}\zeta_n)^k$ and $\max_{k\in\mcl{K}}(\mf{S}_n^{-1}\mf{V}_n^{1/2}\zeta_n)^k$ for a given set of indices $\mcl{K}\subset\{1,\dots,d^2\}$ by simulation. Such a result can be used to construct simultaneous confidence intervals and control the family-wise error rate in multiple testing for entries of $\Theta_Z$; see Sections 2.3--2.4 of \cite{BCCHK2018} for details. 

\section{Simulation study}\label{sec:simulation}

\subsection{Implementation}

In order to implement the proposed estimation procedure, we need to solve the optimization problem in \eqref{f-wglasso}. Among many existing algorithms to solve this problem, we employ the \textsc{glassofast} algorithm of \cite{SC2012}, which is an improved implementation of the popular \textsc{glasso} algorithm of \cite{FHT2008} and implemented in the R package \pck{glassoFast}. 

The remaining problem is how to select the penalty parameter $\lambda$. Following \cite{YL2007,BNS2018}, we select it by minimizing the following formally defined Bayesian information criterion (BIC):
\[
\bic(\lambda):=n\left\{\trace\left(\hat{\Theta}_{Z,\lambda} \hat{\Sigma}_{Z,n}\right)-\log\det\left(\hat{\Theta}_{Z,\lambda}\right)\right\}+(\log n)\sum_{i\leq j}1_{\left\{\hat{\Theta}_{Z,\lambda}^{ij}\neq0\right\}}.
\]
The minimization is carried out by grid search. The grid $\{\lambda_1,\dots,\lambda_m\}$ is constructed analogously to the R package \pck{glmnet} (see Section 2.5 of \cite{FHT2010} for details): First, as the maximum value $\lambda_{\max}$ of the grid, we take the smallest value for which all the off-diagonal entries of $\hat{\Theta}_{Z,\lambda_{\max}}$ are zero: In our case, $\lambda_{\max}$ is set to the maximum modulus of the off-diagonal entries of $\hat{\Sigma}_{Z,n}$ (cf.~\cite[Corollary 1]{WFS2011}). Next, we take a constant $\varepsilon>0$ and set $\lambda_{\min}:=\varepsilon\lambda_{\max}$ as the minimum value of the grid. Finally, we construct the values $\lambda_1,\dots,\lambda_m$ increasing from $\lambda_{\min}$ to $\lambda_{\max}$ on the log scale:
\[
\lambda_i=\exp\left(\log(\lambda_{\min})+\frac{i-1}{m-1}\log(\lambda_{\max}/\lambda_{\min})\right),\qquad i=1,\dots,m.
\]
We use $\varepsilon=\sqrt{(\log d)/n}$ and $m=10$ in our experiments. 

\subsection{Simulation design}

We basically follow the setting of \cite{FFX2016}. We simulate the model \eqref{factor-model} with the following specification: For the factor process $X$, we set $r=3$ and
\begin{align*}
dX^j_t&=\mu_j dt+\sqrt{v^j_t}dW^j_t,&
dv^j_t&=\kappa_j(\theta_j-v^j_t)dt+\eta_j\sqrt{v^j_t}\lpa\rho_jdW^j_t+\sqrt{1-\rho_j^2}d\wt{W}^j_t\rpa,\quad j=1,2,3,
\end{align*}
where $W^1,W^2,W^3,\wt{W}^1,\wt{W}^2,\wt{W}^3$ are independent standard Wiener processes. 
We set $\kappa=(3, 4, 5),\theta=(0.09,0.04,0.06),\eta=(0.3,0.4,0.3),\rho=(-0.6,-0.4,-0.25)$ and $\mu=(0.05,0.03,0.02)$. 
The initial value $v^j_0$ is drawn from the stationary distribution of the process $(v^j_t)_{t\in[0,1]}$, i.e.~the gamma distribution with shape $2\kappa_j\theta_j/\eta_j^2$ and rate $2\kappa_j/\eta_j^2$. 
The entries of the loading $\beta$ are independently drawn as $\beta^{i1}\overset{i.i.d.}{\sim}\mcl{U}[0.25,2.25]$ and $\beta^{i2},\beta^{i3}\overset{i.i.d.}{\sim}\mcl{U}[-0.5,0.5]$ ($\mcl{U}[a,b]$ denotes the uniform distribution on $[a,b]$). 
Finally, as the residual process $Z$, we take a $d$-dimensional Wiener process with covariance matrix $Q$. We consider the following two designs for $Q$:
\begin{description}

\item[Design 1] $Q$ is a block diagonal matrix with 10 blocks of size $(d/10)\times(d/10)$. Each block has diagonal entries independently generated from $\mcl{U}[0.2,0.5]$ and a constant correlation of $0.25$. 

\item[Design 2] We simulate a Chung-Lu random graph $\mcl{G}$ and set $Q:=(\mathsf{E}_d+\bs{D}-\bs{A})$, where $\bs{D}$ and $\bs{A}$ are respectively the degree and adjacent matrices of the random graph $\mcl{G}$. 
Formally, given a weight vector $w\in\mathbb{R}^d$ with $w\geq0$, $\bs{A}$ is defined as a $d\times d$ symmetric random matrix such that all the diagonal entries of $\bs{A}$ are equal to 0 and the off-diagonal upper triangular entries are generated by independent Bernoulli variables so that $P(\bs{A}^{ij}=1)=1-P(\bs{A}^{ij}=0)=w^iw^j/\sum_{k=1}^dw^k$ for $i<j$. Then, $\bs{D}$ is defined as the diagonal matrix such that the $j$-th diagonal entry of $\bs{D}$ is given by $\mf{d}_j(\bs{A})=\sum_{i=1}^d\bs{A}^{ij}$. 
The weight vector $w$ is specified as follows: For every $i=1,\dots,d$, we set $w^i:=c\lpa(i+i_0-1)/d\rpa^{-1/(\alpha-1)}$ with $i_0:=d(c/w_M)^{\alpha-1}$ and $c:=(\alpha-2)/(\alpha-1)$, where we use $\alpha=2.5$ and $w_M=\lfloor d^{0.45}\rfloor$. 

\end{description}
Design 1 is the same one as in \cite{FFX2016}. 
Design 2 is motivated by the recent work of \citet{BBL2018}, which reports that several characteristics of the residual precision matrix of the S\&P 500 assets exhibit power-law behaviors and they are well-described by the power-law partial correlation network model proposed in \cite{BBL2018}; the specification in Design 2 is the same one as in the simulation study of \cite{BBL2018}. 

We observe the processes $Y$ and $X$ at the equidistant sampling times $h/n$, $h=0,1,\dots,n$. 
We set $d=500$ and vary $n$ as $n\in\{78,130,195,390,780\}$. 
We run 10,000 Monte Carlo iterations for each experiment. 

\subsection{Results}

We begin by assessing the estimation accuracy of the proposed estimator in various norms. For comparison, we consider the following 5 different methods to estimate $\Sigma_Y$:
\begin{enumerate}[align=parleft,widest=\textsf{f-wglasso},font=\sffamily,leftmargin=*]

\item[RC] We simply use the realized covariance matrix $\wh{[Y,Y]}_1^n$ defined by \eqref{def:rc} to estimate $\Sigma_Y$. 

\item[glasso] We estimate $\Sigma_Y^{-1}$ by the (unweighted) graphical Lasso based on $\wh{[Y,Y]}_1^n$. Then, $\Sigma_Y$ is estimated by its inverse. 

\item[wglasso] We estimate $\Sigma_Y^{-1}$ by the weighted graphical Lasso based on $\wh{[Y,Y]}_1^n$ (i.e.~the estimator defined by \eqref{wglasso} with $\hat{\Sigma}_n=\wh{[Y,Y]}_1^n$). Then, $\Sigma_Y$ is estimated by its inverse. 

\item[f-glasso] We estimate $\Sigma_Z^{-1}$ by the (unweighted) graphical Lasso based on $\hat{\Sigma}_{Z,n}$ defined by \eqref{def:sigma.z.hat} with $\hat{\Sigma}_{Y,n}=\wh{[Y,Y]}_1^n$ and $\hat{\Sigma}_{X,n}=\wh{[X,X]}_1^n$. Then, $\Sigma_Y$ is estimated by \eqref{def:f-sigma.y} with $\hat{\Theta}_{Z,\lambda}$ being the estimator so constructed. 

\item[f-wglasso] We estimate $\Sigma_Z^{-1}$ by the weighted graphical Lasso based on $\hat{\Sigma}_{Z,n}$ defined by \eqref{def:sigma.z.hat} with $\hat{\Sigma}_{Y,n}=\wh{[Y,Y]}_1^n$ and $\hat{\Sigma}_{X,n}=\wh{[X,X]}_1^n$. Then, $\Sigma_Y$ is estimated by \eqref{def:f-sigma.y} with $\hat{\Theta}_{Z,\lambda}$ being the estimator so constructed. 

\end{enumerate}
In addition, for Design 1, we also consider the estimator proposed in \cite{FFX2016}: Assuming that we know which entries of $\Sigma_Z$ are zero, we estimate $\Sigma_Y$ by $\hat{\beta}_n\hat{\Sigma}_{X,n}\hat{\beta}_n^\top+(\hat{\Sigma}_{Z,n}^{ij}1_{\{\Sigma_Z^{ij}\neq0\}})_{1\leq i,j\leq d}$. We label this method \textsf{f-thr}. 
Since the estimates of \textsf{RC} and \textsf{f-thr} are not always regular, we use their Moore-Penrose generalized inverses to estimate $\Sigma_Y^{-1}$ when they are singular. 
Note that the methods \textsf{glasso} and \textsf{f-glasso} correspond to those proposed in \cite{BNS2018}, while \textsf{wglasso} and \textsf{f-wglasso} are those proposed in this paper. 
We report the simulation results in Tables \ref{table:heston}--\ref{table:clgraph}. 

We first focus on the accuracy of estimating the precision matrix $\Sigma_Y^{-1}$.
The tables reveal the excellent performance of graphical Lasso based methods. In particular, they outperform \textsf{f-thr} in Design 1 except for the case $n=780$ even when we ignore the factor structure of the model. 
Nevertheless, the tables also show apparent benefit to take the factor structure into account in constructing the graphical Lasso type estimators. 
When we compare the weighted graphical Lasso estimators with the unweighted versions, the weighted ones tend to outperform the unweighted ones as $n$ increases, especially when the factor structure is taken into account. This is more pronounced in Design 2. 
It is also worth mentioning that the estimation errors for $\Sigma_Y^{-1}$ in the method \textsf{RC} are greater at $n=390,780$ than those at $n=78,135,390$. This is presumably due to a ``resonance'' effect between the sample size $n$ and dimension $d$ coming from the use of the Moore-Penrose generalized inverse, which is well-known in multivariate analysis (see e.g.~\cite{Hoyle2011}): The estimation error for the precision matrix by the generalized inverse of the sample covariance matrix drastically increases as $n$ approaches $d$. Theoretically, this occurs because the smallest non-zero eigenvalue of the  sample covariance matrix tends to 0 as $n$ approaches $d$. 

Turning to the estimation accuracy for $\Sigma_Y$ in terms of the $\ell_\infty$-norm, we find little advantage to use the graphical Lasso type methods over the realized covariance matrix: \textsf{f-glasso} and \textsf{f-wglasso} tend to outperform \textsf{RC} at small values of $n$, but the differences of the performance become less clear as $n$ increases. In addition, in Design 1 \textsf{f-thr} performs the best in terms of estimating $\Sigma_Y$ at all the values of $n$. 

Next we assess the accuracy of the mixed normal approximation for the de-biased estimator. 
For this purpose, we construct entry-wise confidence intervals for $\Theta_Z$ based on \eqref{eq:point-AN} (with taking the factor structure into account) and evaluate their nominal coverages. Table \ref{table:ci} reports these coverages averaged over the sets $\{(i,j):i\leq j,~\Theta_Z^{ij}=0\}$ and $\{(i,j):i\leq j,~\Theta_Z^{ij}\neq0\}$, respectively. 
We see from the table that the asymptotic approximation perfectly works to construct confidence intervals for zero entries of $\Theta_Z$. By contrast, confidence intervals for non-zero entries of $\Theta_Z$ tend to be over-coverages, especially in Design 1. However, these coverage distortions tend to be moderate at larger values of $n$, which suggests that the normal approximation starts to work for relatively large sample sizes. 

\begin{table}[ht]
\centering
\small
\caption{Estimation accuracy of different methods in Design 1} 
\label{table:heston}
\begin{tabular}{ccrrrrrr}
  & $n$ & \textsf{RC} & \textsf{glasso} & \textsf{wglasso} & \textsf{f-glasso} & \textsf{f-wglasso} & \textsf{f-thr} \\ 
  \hline
 & 78 & 22.431 & 18.857 & 19.083 & 15.122 & 15.130 & 416.197 \\ 
   & 130 & 26.307 & 17.931 & 17.954 & 14.353 & 14.353 & 93.242 \\ 
  $\opnorm{\hat{\Sigma}_Y^{-1}-\Sigma_Y^{-1}}_{\infty}$ & 195 & 45.795 & 17.447 & 17.471 & 13.923 & 13.928 & 50.605 \\ 
   & 390 & 722.381 & 16.687 & 16.678 & 11.306 & 10.806 & 25.335 \\ 
   & 780 & 423.434 & 15.965 & 15.908 & 9.387 & 8.851 & 15.227 \\ 
   \hline
 & 78 & 6.576 & 4.270 & 4.263 & 3.419 & 3.420 & 138.442 \\ 
   & 130 & 6.508 & 3.654 & 3.468 & 3.193 & 3.193 & 28.384 \\ 
  $\opnorm{\hat{\Sigma}_Y^{-1}-\Sigma_Y^{-1}}_{2}$ & 195 & 6.480 & 3.381 & 3.271 & 3.094 & 3.097 & 14.307 \\ 
   & 390 & 203.038 & 3.009 & 3.015 & 2.133 & 2.100 & 6.446 \\ 
   & 780 & 93.354 & 2.788 & 2.855 & 1.782 & 1.693 & 3.562 \\ 
   \hline
 & 78 & 0.361 & 0.432 & 0.441 & 0.351 & 0.351 & 0.347 \\ 
   & 130 & 0.279 & 0.311 & 0.296 & 0.281 & 0.281 & 0.268 \\ 
  $\|\hat{\Sigma}_Y-\Sigma_Y\|_{\ell_{\infty}}$ & 195 & 0.227 & 0.255 & 0.250 & 0.241 & 0.241 & 0.219 \\ 
   & 390 & 0.160 & 0.181 & 0.189 & 0.166 & 0.169 & 0.154 \\ 
   & 780 & 0.112 & 0.130 & 0.143 & 0.118 & 0.119 & 0.108 \\ 
   \hline
\end{tabular}\vspace{2mm}

\parbox{14cm}{\footnotesize
\textit{Note}. \textsf{RC}: realized covariance matrix; \textsf{glasso}: graphical Lasso; \textsf{wglasso}: weighted graphical Lasso; \textsf{f-glasso}: graphical Lasso with taking the factor structure into account; \textsf{f-wglasso}: weighted graphical Lasso with taking the factor structure into account; \textsf{f-thr}: location-based thresholding with taking the factor structure into account (the method of \cite{FFX2016}). 
The results are based on 10,000 Monte Carlo iterations. 
} 
\end{table}

\begin{table}[ht]
\centering
\small
\caption{Estimation accuracy of different methods in Design 2} 
\label{table:clgraph}
\begin{tabular}{ccrrrrr}
  & $n$ & \textsf{RC} & \textsf{glasso} & \textsf{wglasso} & \textsf{f-glasso} & \textsf{f-wglasso} \\ 
  \hline
 & 78 & 47.934 & 43.144 & 43.055 & 35.347 & 35.263 \\ 
   & 130 & 48.266 & 43.166 & 41.750 & 34.767 & 34.284 \\ 
  $\opnorm{\hat{\Sigma}_Y^{-1}-\Sigma_Y^{-1}}_{\infty}$ & 195 & 50.049 & 42.806 & 40.571 & 34.154 & 32.835 \\ 
   & 390 & 338.847 & 41.060 & 37.801 & 33.100 & 29.934 \\ 
   & 780 & 401.447 & 38.886 & 34.961 & 32.163 & 23.121 \\ 
   \hline
 & 78 & 17.805 & 13.557 & 13.522 & 7.857 & 7.843 \\ 
   & 130 & 17.798 & 13.543 & 12.628 & 7.954 & 7.866 \\ 
  $\opnorm{\hat{\Sigma}_Y^{-1}-\Sigma_Y^{-1}}_{2}$ & 195 & 17.752 & 13.319 & 11.630 & 8.006 & 7.742 \\ 
   & 390 & 87.239 & 12.296 & 9.888 & 8.059 & 7.416 \\ 
   & 780 & 55.619 & 11.189 & 8.522 & 8.065 & 6.072 \\ 
   \hline
 & 78 & 0.669 & 0.723 & 0.721 & 0.632 & 0.631 \\ 
   & 130 & 0.509 & 0.678 & 0.572 & 0.489 & 0.481 \\ 
  $\|\hat{\Sigma}_Y-\Sigma_Y\|_{\ell_{\infty}}$ & 195 & 0.412 & 0.567 & 0.470 & 0.403 & 0.390 \\ 
   & 390 & 0.289 & 0.298 & 0.339 & 0.282 & 0.273 \\ 
   & 780 & 0.203 & 0.198 & 0.252 & 0.197 & 0.192 \\ 
   \hline
\end{tabular}\vspace{2mm}

\parbox{14cm}{\footnotesize
\textit{Note}. \textsf{RC}: realized covariance matrix; \textsf{glasso}: graphical Lasso; \textsf{wglasso}: weighted graphical Lasso; \textsf{f-glasso}: graphical Lasso with taking the factor structure into account;  \textsf{f-wglasso}: weighted graphical Lasso with taking the factor structure into account. 
The results are based on 10,000 Monte Carlo iterations. 
} 
\end{table}


\begin{table}[ht]
\centering
\small
\caption{Average coverages of entry-wise confidence intervals} 
\label{table:ci}
\begin{tabular}{cccrrrrr}
  & & & \multicolumn{2}{c}{Design 1} &  & \multicolumn{2}{c}{Design 2} \\
 &  & $n$ & 95\% & 99\% &  & 95\% & 99\% \\ 
   \hline
 &  & 78 & 95.21 & 99.04 &  & 95.22 & 99.04 \\ 
   &  & 130 & 95.13 & 99.03 &  & 95.13 & 99.03 \\ 
  $\Theta_Z^{ij}=0$ &  & 195 & 95.09 & 99.02 &  & 95.09 & 99.02 \\ 
   &  & 390 & 95.04 & 99.01 &  & 95.05 & 99.01 \\ 
   &  & 780 & 95.02 & 99.00 &  & 95.02 & 99.01 \\ 
   \hline
 &  & 78 & 99.33 & 99.87 &  & 95.16 & 99.03 \\ 
   &  & 130 & 99.82 & 99.96 &  & 95.90 & 99.18 \\ 
  $\Theta_Z^{ij}\neq0$ &  & 195 & 99.97 & 99.99 &  & 96.36 & 99.26 \\ 
   &  & 390 & 96.00 & 99.20 &  & 96.65 & 99.33 \\ 
   &  & 780 & 96.09 & 99.22 &  & 96.41 & 99.27 \\ 
   \hline
\end{tabular}\vspace{2mm}

\parbox{11cm}{\footnotesize 
This table reports the average coverages of entry-wise confidence intervals for the residual precision matrix $\Theta_Z$ over the sets $\{(i,j):i\leq j,~\Theta_Z^{ij}=0\}$ and $\{(i,j):i\leq j,~\Theta_Z^{ij}\neq0\}$, respectively. 
The confidence intervals are constructed based on the normal approximation \eqref{eq:point-AN}. 
The results are based on 10,000 Monte Carlo iterations. 
} 
\end{table}

\clearpage

\section{Conclusion}

In this paper we have developed a generic asymptotic theory to estimate the high-dimensional precision matrix of high-frequency data using the weighted graphical Lasso. We have shown that the consistency of the weighted graphical Lasso estimator in matrix operator norms follows from the consistency of the initial estimator in the $\ell_\infty$-norm, while the asymptotic mixed normality of its de-biased version follows from that of the initial estimator, where the asymptotic mixed normality has been formulated appropriately for the high-dimensional setting considered here. Our theory also encompasses a situation where a known factor structure is present in the data. In such a situation, we have applied the weighted graphical Lasso to the residual process obtained after removing the effect of factors. 

We have applied the developed theory to the concrete situation where we can use the realized covariance matrix as the initial covariance estimator. We have derived the desirable asymptotic mixed normality of the realized covariance matrix by an application of the recent high-dimensional central limit theorem obtained in \cite{Koike2018sk}, where Malliavin calculus resolves the main theoretical difficulties caused by the high-dimensionality. As a consequence, we have obtained a feasible asymptotic distribution theory to conduct inference for entries of the precision matrix. A Monte Carlo study has shown the good finite sample performance of our asymptotic theory. 

A natural direction for future work is to apply the developed theory to a more complex situation where the process is asynchronously observed with noise and/or jumps. To accomplish this purpose, we need to establish the high-dimensional asymptotic mixed normality of relevant covariance estimators. 
 

\appendix

\section*{Appendix: Proofs}

\section{Matrix inequalities}

This appendix collects some elementary (but less trivial) inequalities for matrices used in the proofs of the main results. 
%
\begin{lemma}\label{lemma:poincare}
Let $A\in\mcl{S}_d$. Then $\Lambda_{\min}(A)\leq A^{ii}\leq\Lambda_{\max}(A)$ for every $i=1,\dots,d$.
\end{lemma}

\begin{proof}
See Theorem 14 in \cite[Chapter 11]{MN1988}. 
\end{proof}

\begin{lemma}\label{lemma:rayleigh}
Let $A\in\mcl{S}^+_d$ and $B\in\mathbb{R}^{d\times r}$. Then $\Lambda_{\max}(B^\top AB)\leq\Lambda_{\max}(B^\top B)\Lambda_{\max}(A)$ and $\Lambda_{\min}(B^\top AB)\geq\Lambda_{\min}(B^\top B)\Lambda_{\min}(A)$. 
\end{lemma}

\begin{proof}
Let $x$ be an eigenvector associated with $\Lambda_{\max}(A)$ such that $\|x\|_{\ell_2}=1$. Then, by Theorem 4 in \cite[Chapter 11]{MN1988} we have $\Lambda_{\max}(B^\top AB)=x^\top B^\top ABx\leq\Lambda_{\max}(A)x^\top B^\top Bx\leq \Lambda_{\max}(A)\Lambda_{\max}(B^\top B)$. So we obtain the first inequality. The second one can be shown analogously.  
\end{proof}

\begin{lemma}\label{lemma:weyl}
Let $A,B\in\mcl{S}_d$. Then $|\Lambda_{\max}(A)-\Lambda_{\max}(B)|\vee|\Lambda_{\min}(A)-\Lambda_{\min}(B)|\leq\opnorm{A-B}_2$. 
\end{lemma}

\begin{proof}
Noting the identity $\opnorm{C}_2=\Lambda_{\max}(C)\vee(-\Lambda_{\min}(C))$ holding for any symmetric matrix $C$, the desired result follows from Weyl's inequality (cf.~Corollary 4.3.15 in \cite{HJ2013}).
\end{proof}

\begin{lemma}\label{lemma:op-sp}
For any $A\in\mcl{S}_d$, $\opnorm{A}_1=\opnorm{A}_\infty\leq\sqrt{\mf{d}(A)}\opnorm{A}_2$.
\end{lemma}

\begin{proof}
This is a straightforward consequence of the Schwarz inequality. 
\end{proof}

\begin{lemma}\label{inv-diff-op}
Let $A,B\in\mathbb{R}^{r\times r}$. If $A$ is invertible and $\opnorm{A^{-1}(B-A)}_w<1$ for some $w\in[1,\infty]$, $B$ is invertible and
\[
\opnorm{B^{-1}-A^{-1}}_w\leq\frac{\opnorm{A^{-1}}_w\opnorm{A^{-1}(B-A)}_w}{1-\opnorm{A^{-1}(B-A)}_w}.
\]
\end{lemma}

\begin{proof}
See pages 381--382 of \cite{HJ2013}.
\end{proof}

\begin{lemma}\label{lemma:ogihara}
Let $A\in\mcl{S}_r$ and $B,C\in\mathbb{R}^{d\times r}$. Then 
\[
\|BAC^\top\|_{\ell_\infty}\leq \opnorm{A}_2\lpa\max_{1\leq i\leq d}\|B^{i\cdot}\|_{\ell_2}\rpa\lpa\max_{1\leq j\leq d}\|C^{j\cdot}\|_{\ell_2}\rpa
\leq r\opnorm{A}_2\|B\|_{\ell_\infty}\|C\|_{\ell_\infty}.
\]
\end{lemma}

\begin{proof}
This result has essentially been shown in \cite[Lemma A.7]{Ogihara2018}. 
Since $A$ is symmetric, there is an orthogonal matrix $U\in\mathbb{R}^{r\times r}$ such that $\Lambda:=U^\top AU$ is a diagonal matrix. Now, for any $i,j\in\{1,\dots,d\}$,
\begin{align*}
|(BAC^\top)^{ij}|
&=|(B^{i\cdot})^\top AC^{j\cdot}|
=|(B^{i\cdot})^\top U\Lambda U^\top C^{j\cdot}|
=\labs\sum_{k=1}^r\Lambda^{kk}((B^{i\cdot})^\top U)^k(U^\top C^{j\cdot})^k\rabs\\
&\leq\max_{1\leq k\leq r}|\Lambda^{kk}|\|U^\top B^{i\cdot}\|_{\ell_2}\|UC^{j\cdot}\|_{\ell_2}
=\opnorm{A}_2\|B^{i\cdot}\|_{\ell_2}\|C^{j\cdot}\|_{\ell_2}
\leq r\opnorm{A}_2\|B\|_{\ell_\infty}\|C\|_{\ell_\infty}.
\end{align*}
This yields the desired result.
\end{proof}

\begin{lemma}\label{lemma:schwarz}
Let $A,B,C\in\mathbb{R}^{d\times d}$. Then, for any $i,j=1,\dots,d$,  
\[
|(BAC^\top)^{ij}|^2\leq\lpa\sum_{k=1}^d|B^{ik}|\rpa\lpa\sum_{l=1}^d|C^{il}|\rpa\sum_{k,l=1}^d|B^{ik}C^{jl}|(A^{kl})^2.
\] 
\end{lemma}

\begin{proof}
This is a straightforward consequence of the Schwarz inequality. 
\end{proof}

\section{Proofs for Section \ref{sec:main}}

\subsection{Proof of Proposition \ref{glasso:rate}}

The following result has essentially been proven in \cite{JvdG2018} and gives an estimate for the ``deterministic part'' of oracle inequalities for graphical Lasso type estimators. 
\begin{proposition}\label{oracle}
Let $A_0,A\in\mcl{S}_d$ and assume $\|A-A_0\|_{\ell_\infty}\leq\lambda_0$ for some $\lambda_0>0$. Assume also that there are numbers $L>1$ and $\lambda>0$ such that
$L^{-1}\leq\Lambda_\mathrm{min}(A_0)\leq\Lambda_\mathrm{max}(A_0)\leq L$, 
$2\lambda_0\leq\lambda\leq(8Lc_L)^{-1}$ and 
$8c_L^2s\lambda^2+2c_L\lambda_0^2\|\diag(A)-\diag(A_0)\|_{\ell_2}^2\leq\lambda_0/(2L)$, 
where $s:=s(B_0)$ and $c_L:=8L^2$. Set $B_0:=A_0^{-1}$. Then, for any $B\in\mcl{S}_d^{++}$ satisfying
\begin{equation}\label{eq:convex}
\trace\left(BA\right)-\log\det\left(B\right)+\lambda\|B^{-}\|_{\ell_1}
\leq\trace\left(B_0A\right)-\log\det\left(B_0\right)+\lambda\|B_0^{-}\|_{\ell_1},
\end{equation}
it holds that
\[
\|B-B_0\|_{\ell_2}^2/c_L+\lambda\|B^--B_0^-\|_{\ell_1}\leq8c_L^2s\lambda^2+2c_L\lambda_0^2\|\diag(A)-\diag(A_0)\|_{\ell_2}^2.
\]
\end{proposition}

We first prove Proposition \ref{oracle} under an additional assumption:
\begin{lemma}\label{lemma:oracle}
Proposition \ref{oracle} holds true if we additionally have $\|B-B_0\|_{\ell_2}\leq1/(2L)$.   
\end{lemma}

\begin{proof}
Set $\Delta=B-B_0$. By assumption we have $\|\Delta\|_{\ell_2}\leq1/(2L)$, so Lemma 2 in \cite{JvdG2018} implies that
\[
\mathcal{E}(\Delta):=\trace\left(\Delta A_0\right)-\left\{\log\det\left(\Delta+B_0\right)-\log\det\left(B_0\right)\right\}
\]
is well-defined and we have
\begin{equation}\label{margin}
\mathcal{E}(\Delta)\geq c\|\Delta\|_{\ell_2},
\end{equation}
where $c=c_L^{-1}$. 
Moreover, \eqref{eq:convex} yields
\begin{align}
\mathcal{E}(\Delta)+\lambda\|B^{-}\|_{\ell_1}
&=-\trace\left(\Delta(A-A_0)\right)
+\left\{\trace\left(BA\right)-\log\det\left(B\right)+\lambda\|B^{-}\|_{\ell_1}\right\}
-\trace\left(B_0A\right)+\log\det\left(B_0\right)\nonumber\\
&\leq-\trace\left(\Delta(A-A_0)\right)
+\left\{\trace\left(B_0A\right)-\log\det\left(B_0\right)+\lambda\|B_0^{-}\|_{\ell_1}\right\}
-\trace\left(B_0A\right)+\log\det\left(B_0\right)\nonumber\\
&=-\trace\left(\Delta(A-A_0)\right)+\lambda\|B_0^{-}\|_{\ell_1}.\label{basic}
\end{align}
Now, note that $\trace(A_1B_1)=\trace(A_1^-B_1^-)+\trace(\diag(A_1)\diag(B_1))$ and $|\trace(A_1B_1)|\leq\|A_1\circ B_1\|_{\ell_1}$ for any $A_1,B_1\in\mathbb{R}^{d\times d}$. Thus, we infer that
\begin{align*}
|\trace\left(\Delta(A-A_0)\right)|
&\leq\|\Delta^-\|_{\ell_1}\|A^--A_0^-\|_{\ell_\infty}+\|\diag(\Delta)\|_{\ell_2}\|\diag(A)-\diag(A_0)\|_{\ell_2}\\
&\leq\lambda_0\|\Delta^-\|_{\ell_1}+\|\diag(A)-\diag(A_0)\|_{\ell_2}\|\diag(\Delta)\|_{\ell_2},
\end{align*}
where we use $\|A-A_0\|_{\ell_\infty}\leq\lambda_0$ in the last line. Combining this with \eqref{margin}--\eqref{basic}, we conclude that
\[
c\|\Delta\|_{\ell_2}^2+\lambda\|B^{-}\|_{\ell_1}
\leq\lambda_0\|\Delta^-\|_{\ell_1}+\|\diag(A)-\diag(A_0)\|_{\ell_2}\|\diag(\Delta)\|_{\ell_2}+\lambda\|B_0^{-}\|_{\ell_1}.
\]
Let $S:=S(B_0)$. Also, for a subset $I$ of $\{1,\dots,d\}^2$ and a $d\times d$ matrix $U$, we define the $d\times d$ matrix $U_I=(U_I^{ij})_{1\leq i,j\leq d}$ by $U_I^{ij}=U^{ij}1_{\{(i,j)\in I\}}$. 
Then, by definition and assumption, we have
$\|B^{-}\|_{\ell_1}=\|B_{S}^{-}\|_{\ell_1}+\|B_{S^c}^{-}\|_{\ell_1}$, 
$\|\Delta^-\|_{\ell_1}=\|\Delta_S^-\|_{\ell_1}+\|B_{S^c}^-\|_{\ell_1}$, 
$\|B_0^{-}\|_{\ell_1}\leq\|\Delta_S^-\|_{\ell_1}+\|B_{S}^{-}\|_{\ell_1}$, 
$\lambda\geq2\lambda_0$, 
so we deduce
\[
c\|\Delta\|_{\ell_2}^2+\frac{\lambda}{2}\|B_{S^c}^{-}\|_{\ell_1}
\leq\frac{3\lambda}{2}\|\Delta^-_S\|_{\ell_1}+\|\diag(A)-\diag(A_0)\|_{\ell_2}\|\diag(\Delta)\|_{\ell_2}.
\]
Consequently, we obtain
\begin{align*}
2c\|\Delta\|_{\ell_2}^2+\lambda\|\Delta^{-}\|_{\ell_1}
&=2c\|\Delta\|_{\ell_2}^2+\lambda(\|B_{S^c}^{-}\|_{\ell_1}+\|\Delta_S^{-}\|_{\ell_1})\\
&\leq4\lambda\|\Delta^-_S\|_{\ell_1}+2\|\diag(A)-\diag(A_0)\|_{\ell_2}\|\diag(\Delta)\|_{\ell_2}\\
&\leq4\lambda\sqrt{s}\|\Delta^-_S\|_{\ell_2}+2\|\diag(A)-\diag(A_0)\|_{\ell_2}\|\diag(\Delta)\|_{\ell_2}~(\because\text{Schwarz inequality})\\
&\leq8s\lambda^2/c^2+c\|\Delta^-_S\|_{\ell_2}^2/2+2\|\diag(A)-\diag(A_0)\|_{\ell_2}^2/c+c\|\diag(\Delta)\|_{\ell_2}^2/2,
\end{align*}
where we use the inequality $xy\leq(x^2+y^2)/2$ in the last line. Since $\|\Delta\|_{\ell_2}^2=\|\diag(\Delta)\|_{\ell_2}^2+\|\Delta^-\|_{\ell_2}^2$, we conclude that
\[
c\|\Delta\|_{\ell_2}^2+\lambda\|\Delta^{-}\|_{\ell_1}
\leq8s\lambda^2/c^2+2\|\diag(A)-\diag(A_0)\|_{\ell_2}^2/c,
\]
which completes the proof.
\end{proof}

\begin{proof}[Proof of Proposition \ref{oracle}]
Thanks to Lemma \ref{lemma:oracle}, it suffices to prove $\|B-B_0\|_{\ell_2}\leq1/(2L)$. 

Set $\tilde{B}=\alpha B+(1-\alpha)B_0$ with $\alpha=M/(M+\|B-B_0\|_{\ell_2})$ and $M=1/(2L)$. By definition we have $\|\tilde{B}-B_0\|_{\ell_2}\leq M=1/(2L)$. Moreover, \eqref{eq:convex} and the convexity of the loss function imply that
\begin{equation*}
\trace(\tilde{B}A)-\log\det(\tilde{B})+\lambda\|B^{-}\|_{\ell_1}
\leq\trace\left(B_0A\right)-\log\det\left(B_0\right)+\lambda\|B_0^{-}\|_{\ell_1}.
\end{equation*}
Therefore, we can apply Lemma \ref{lemma:oracle} with replacing $B$ by $\tilde{B}$, and thus we obtain
\[
\|\tilde{B}-B_0\|_{\ell_2}^2/c_L+\lambda\|\tilde{B}^--B_0^-\|_{\ell_1}\leq8c_L^2s\lambda^2+2c_L\lambda_0^2\|\diag(A)-\diag(\Sigma)\|_{\ell_2}^2.
\]
In particular, we have
\[
\|\tilde{B}-B_0\|_{\ell_2}^2\leq c_L\lambda_0/(2L)\leq1/(16L^2),
\]
so we obtain $\|\tilde{B}-B_0\|_{\ell_2}\leq1/(4L)=M/2$. By construction this yields $\|B-B_0\|_{\ell_2}\leq M=1/(2L)$, which completes the proof. 
\end{proof}

\begin{proof}[Proof of Proposition \ref{glasso:rate}]
Thanks to Lemma 7.2 of \cite{CLZ2016}, it suffices to consider the case $w=1$. 

For any $L,n\in\mathbb{N}$, we define the set $\Omega_{n,L}\subset\Omega$ by
\begin{multline*}
\Omega_{n,L}:=\{\|\hat{R}_n-R_Y\|_{\ell_\infty}\leq\lambda_n/2\}\cap\{L^{-1}\leq\Lambda_{\min}(R_Y)\leq \Lambda_{\max}(R_Y)\leq L\}\\
\cap\{s(K_Y)\leq Ls_n\}
\cap\{8c_{L}^2Ls_n\lambda_n>1/(4L)\},
\end{multline*}
where $c_{L}:=8L^2$. 
Then we have
\[
\lim_{L\to\infty}\limsup_{n\to\infty}P(\Omega_{n,L}^c)=0.
\]
In fact, noting that Lemma \ref{lemma:poincare} and \ref{ass:eigen} yield 
\begin{equation}\label{bdd:vol}
\max_{1\leq j\leq d}\lpa\Sigma_Y^{jj}+1/\Sigma_Y^{jj}\rpa=O_p(1),
\end{equation} 
\ref{ass:eigen}--\ref{ass:sparse}, \ref{ass:est} and Lemma \ref{lemma:rayleigh} imply that $\lambda_n^{-1}\|\hat{R}_n-R_Y\|_{\ell_\infty}=o_p(1)$, $s(K_Y)=O_p(s_n)$ and $\Lambda_{\max}(R_Y)+1/\Lambda_{\min}(R_Y)=O_p(1)$. 
Finally, \ref{ass:rate} yields $\lim_{n\to\infty}P(8Ls_n\lambda_n>1/(4L))=0$ for all $L$. 
Now, note that $\lambda_n\leq1/(16Lc_{L}^2)\leq(8Lc_{L})^{-1}$ on the set $\Omega_{n,L}$. Therefore, applying Proposition \ref{oracle} with $\lambda:=\lambda_n$ and $\lambda_0:=\lambda_n/2$, for any fixed $L$ we have 
\[
\|\hat{K}_{\lambda_n}-K_Y\|_{\ell_2}^2/c_{L}+\lambda_n\|\hat{K}_{\lambda_n}^--K_Y^-\|_{\ell_1}\leq8c_{L}^2s_n\lambda_n^2\qquad\text{on $\Omega_{n,L}$}.
\]
Consequently, we obtain
\[
\limsup_{n\to\infty}P\left(\|\hat{K}_{\lambda_n}-K_Y\|_{\ell_2}>64L^3\sqrt{s_n}\lambda_n\right)
\leq\limsup_{n\to\infty}P(\Omega_{n,L}^c)
\]
and
\[
\limsup_{n\to\infty}P\left(\|\hat{K}_{\lambda_n}^--K_Y^-\|_{\ell_1}>512L^4s_n\lambda_n\right)
\leq\limsup_{n\to\infty}P(\Omega_{n,L}^c).
\]
Therefore, we conclude that
\[
\limsup_{M\to\infty}\limsup_{n\to\infty}P\left(\|\hat{K}_{\lambda_n}-K_Y\|_{\ell_2}>M\sqrt{s_n}\lambda_n\right)
\leq\limsup_{L\to\infty}\limsup_{n\to\infty}P(\Omega_{n,L}^c)=0
\]
and
\[
\limsup_{M\to\infty}\limsup_{n\to\infty}P\left(\|\hat{K}_{\lambda_n}^--K_Y^-\|_{\ell_1}>Ms_n\lambda_n\right)
\leq\limsup_{L\to\infty}\limsup_{n\to\infty}P(\Omega_{n,L}^c)=0,
\]
which yields $\|\hat{K}_{\lambda_n}-K_Y\|_{\ell_2}=O_p(\sqrt{s_n}\lambda_n)$ and $\|\hat{K}_{\lambda_n}^--K_Y^-\|_{\ell_1}=O_p(s_n\lambda_n)$. In particular, we obtain the first convergence of \eqref{rate:K}. Moreover, since we have
\begin{align*}
\opnorm{\hat{K}_{\lambda_n}-K_Y}_1
&\leq\|\diag(\hat{K}_{\lambda_n})-\diag(K_Y)\|_{\ell_\infty}+\|\hat{K}_{\lambda_n}^--K_Y^-\|_{\ell_1}\\
&\leq\|\hat{K}_{\lambda_n}-K_Y\|_{\ell_2}+\|\hat{K}_{\lambda_n}^--K_Y^-\|_{\ell_1},
\end{align*}
we also obtain the second convergence of \eqref{rate:K}. 

Now we prove \eqref{rate:Theta}. 
First, \eqref{bdd:vol} and \ref{ass:est} yield $\lambda_n^{-1}\opnorm{\hat{ V}_n- V}_1\to^p0$. 
Since $\opnorm{ V}_1=O_p(1)$, $\opnorm{K_Y}_1=O_p(s_n)$ and $\lambda_n=o_p(1)$ by \eqref{bdd:vol}, \ref{ass:sparse} and \ref{ass:rate}, we obtain $\opnorm{\hat{ V}_n}_1=O_p(1)$ and $\opnorm{\hat{K}_{\lambda_n}}_1=O_p(s_n)$. Since 
\begin{align*}
&\opnorm{\hat{\Theta}_{\lambda_n}-\Theta_Y}_1
=\opnorm{\hat{ V}_n\hat{K}_{\lambda_n}\hat{ V}_n- VK_Y V}_1\\
&\leq\opnorm{\hat{ V}_n- V}_1\opnorm{\hat{K}_{\lambda_n}}_1\opnorm{\hat{ V}_n}_1
+\opnorm{ V}_1\opnorm{\hat{K}_{\lambda_n}-K_Y}_1\opnorm{\hat{ V}_n}_1
+\opnorm{ V}_1\opnorm{K_Y}_1\opnorm{\hat{ V}_n- V}_1,
\end{align*}
we obtain the first convergence of \eqref{rate:Theta}. 
Next, since $\opnorm{\hat{\Theta}_{\lambda_n}-\Theta_Y}_2=o_p(1)$ by the above result, \ref{ass:eigen} and Lemma \ref{lemma:weyl} yield $\opnorm{\hat{\Theta}_{\lambda_n}^{-1}}_2=\Lambda_{\min}(\hat{\Theta}_{\lambda_n})^{-1}=O_p(1)$. Since $\opnorm{\hat{\Theta}_{\lambda_n}^{-1}-\Sigma_Y}_2\leq\opnorm{\hat{\Theta}_{\lambda_n}^{-1}}_2\opnorm{\Theta_Y-\hat{\Theta}_{\lambda_n}}_2\opnorm{\Sigma_Y}_2$, we obtain the second convergence of \eqref{rate:Theta}. 
\end{proof}

\subsection{Proof of Lemma \ref{prop:AL}}

First, by Proposition 14.4.3 of \cite{Lange2013} there is a (not necessarily measurable) $d\times d$ random matrix $\hat{Z}_n$ such that
\[
\hat{\Sigma}_n-\hat{\Theta}_{\lambda_n}^{-1}+\lambda_n \hat{ V}_n\hat{Z}_n\hat{ V}_n=0,\qquad
\|\hat{Z}_n\|_{\ell_\infty}\leq1,
\]
and $\hat{Z}_n^{ij}=\sign(\hat{\Theta}_{\lambda_n}^{ij})$ if $\hat{\Theta}_{\lambda_n}^{ij}\neq0$. Consequently, it holds that
\[
\hat{\Sigma}_n\hat{\Theta}_{\lambda_n}-\mathsf{E}_d+\lambda_n \hat{ V}_n\hat{Z}_n\hat{ V}_n\hat{\Theta}_{\lambda_n}=0.
\]
Therefore, we have
\begin{align*}
&\hat{\Theta}_{\lambda_n}-\Theta_Y+\Theta_Y(\hat{\Sigma}_n-\Sigma_Y)\Theta_Y\\
&=\hat{\Theta}_{\lambda_n}-\Theta_Y+\Theta_Y(\hat{\Sigma}_n-\Sigma_Y)\hat{\Theta}_{\lambda_n}-\Theta_Y(\hat{\Sigma}_n-\Sigma_Y)(\hat{\Theta}_{\lambda_n}-\Theta_Y)\\
&=\hat{\Theta}_{\lambda_n}-\Theta_Y+\Theta_Y(\mathsf{E}_d-\lambda_n \hat{ V}_n\hat{Z}_n\hat{ V}_n\hat{\Theta}_{\lambda_n}-\Sigma_Y\hat{\Theta}_{\lambda_n})-\Theta_Y(\hat{\Sigma}_n-\Sigma_Y)(\hat{\Theta}_{\lambda_n}-\Theta_Y)\\
&=-\lambda_n\Theta_Y\hat{ V}_n\hat{Z}_n\hat{ V}_n\hat{\Theta}_{\lambda_n}-\Theta_Y(\hat{\Sigma}_n-\Sigma_Y)(\hat{\Theta}_{\lambda_n}-\Theta_Y)\\
&=\lambda_n(\hat{\Theta}_{\lambda_n}-\Theta_Y)\hat{ V}_n\hat{Z}_n\hat{ V}_n\hat{\Theta}_{\lambda_n}-(\hat{\Theta}_{\lambda_n}-\hat{\Theta}_{\lambda_n}\hat{\Sigma}_n\hat{\Theta}_{\lambda_n})-\Theta_Y(\hat{\Sigma}_n-\Sigma_Y)(\hat{\Theta}_{\lambda_n}-\Theta_Y),
\end{align*}
so we obtain
\begin{align*}
&\|\hat{\Theta}_{\lambda_n}-\Theta_Y-\Gamma_n+\Theta_Y(\hat{\Sigma}_n-\Sigma_Y)\Theta_Y\|_{\ell_\infty}\\
&\leq\lambda_n\opnorm{\hat{\Theta}_{\lambda_n}-\Theta_Y}_\infty\|\hat{ V}_n\hat{Z}_n\hat{ V}_n\hat{\Theta}_{\lambda_n}\|_{\ell_\infty}
+\opnorm{\Theta_Y}_\infty\|\hat{\Sigma}_n-\Sigma_Y\|_{\ell_\infty}\opnorm{\hat{\Theta}_{\lambda_n}-\Theta_Y}_{\ell_\infty}\\
&\leq\lambda_n\opnorm{\hat{\Theta}_{\lambda_n}-\Theta_Y}_\infty\opnorm{\hat{ V}_n}_\infty^2\opnorm{\hat{\Theta}_{\lambda_n}}_{\infty}
+\opnorm{\Theta_Y}_\infty\|\hat{\Sigma}_n-\Sigma_Y\|_{\ell_\infty}\opnorm{\hat{\Theta}_{\lambda_n}-\Theta_Y}_{\ell_\infty}.
\end{align*}
Now the desired result follows from Proposition \ref{glasso:rate} and Lemma \ref{lemma:op-sp}.\hfil\qed

\subsection{Proof of Proposition \ref{prop:AMN}}

In the light of Lemma 3.1 of \cite{Koike2018sk}, it is enough to prove 
\[
\sqrt{\log (m+1)}\left\|J_n\vectorize\left(\hat{\Theta}_{\lambda_n}-\Theta_Y-\Gamma_n\right)-\wt{J}_n\vectorize\left(\hat{\Sigma}_n-\Sigma\right)\right\|_{\ell_\infty}\to^p0
\]
as $n\to\infty$. Note that $\vectorize(ABC)=(C^\top\otimes A)\vectorize(B)$ for any $d\times d$ matrices $A,B,C$ (cf.~Theorem 2 in \cite[Chapter 2]{MN1988}). Thus, we obtain the desired result once we prove
\[
\sqrt{\log (m+1)}\left\|J_n\left(\hat{\Theta}_{\lambda_n}-\Theta_Y-\Gamma_n+\Theta_Y\left(\hat{\Sigma}_n-\Sigma\right)\Theta_Y\right)\right\|_{\ell_\infty}\to^p0
\]
as $n\to\infty$. This follows from Lemma \ref{prop:AL} and the assumptions of the proposition.\hfill\qed 

\section{Proofs for Section \ref{sec:factor}}

\subsection{Proof of Proposition \ref{factor:rate}}

Set $\Omega_n:=\{\opnorm{\Sigma_X^{-1}(\hat{\Sigma}_{X,n}-\Sigma_X)}_2\leq1/2\}$.
\begin{lemma}\label{inv-sigma-f}
Under the assumptions of Proposition \ref{factor:rate}, we have the following results:
\begin{enumerate}[label=(\alph*)]

\item\label{sigma.f-inv} On the event $\Omega_n$, $\hat{\Sigma}_{X,n}$ is invertible and 
$
\opnorm{\hat{\Sigma}_{X,n}^{-1}-\Sigma_X^{-1}}_2\leq2\opnorm{\Sigma_X^{-1}}_2\opnorm{\Sigma_X^{-1}(\hat{\Sigma}_{X,n}-\Sigma_X)}_2.
$

\item\label{sigma.f-sp} $\lambda_n^{-1}\opnorm{\hat{\Sigma}_{X,n}-\Sigma_X}_2=o_p(r)$ and $\opnorm{\hat{\Sigma}_{X,n}}_2=O_p(r)$ as $n\to\infty$. 

\item\label{omega.n} $P(\Omega_n)\to1$ as $n\to\infty$.  

\end{enumerate}
\end{lemma}

\begin{proof}
\ref{sigma.f-inv} is a direct consequence of Lemma \ref{inv-diff-op}. 
\ref{sigma.f-sp} follows from \ref{ass:eigen.f}, \ref{f-ass:rate} and the inequalities $\opnorm{\hat{\Sigma}_{X,n}-\Sigma_X}_2\leq r\|\hat{\Sigma}_{X,n}-\Sigma_X\|_{\ell_\infty}$ and $\opnorm{\Sigma_X}_2\leq r\|\Sigma_X\|_{\ell_\infty}$. 
\ref{omega.n} follows from the inequality $\opnorm{\Sigma_X^{-1}(\hat{\Sigma}_{X,n}-\Sigma_X)}_2\leq r\opnorm{\Sigma_X^{-1}}_2\|\hat{\Sigma}_{X,n}-\Sigma_X\|_{\ell_\infty}$. 
\end{proof}

\begin{lemma}\label{lemma:approx-z}
Under the assumptions of Proposition \ref{factor:rate}, $\lambda_n^{-2}\|\hat{\Sigma}_{Z,n}-\breve{\Sigma}_{Z,n}\|_{\ell_\infty}=o_p(r)$ as $n\to\infty$. 
\end{lemma}

\begin{proof}
Since $\hat{\beta}_n=\hat{\Sigma}_{YX,n}\hat{\Sigma}_{X,n}^{-1}$ on the event $\Omega_n$, we have
\begin{align*}
\hat{\Sigma}_{Z,n}-\breve{\Sigma}_{Z,n}
&=-\hat{\beta}_n\hat{\Sigma}_{X,n}\hat{\beta}_n^\top
+\hat{\Sigma}_{YX,n}\beta^\top+\beta\hat{\Sigma}_{YX,n}^\top-\beta\hat{\Sigma}_{X,n}\beta^\top\\
&=-\hat{\Sigma}_{YX,n}(\hat{\beta}_n-\beta)^\top
+\beta\hat{\Sigma}_{X,n}(\hat{\beta}_n-\beta)^\top\\
&=-(\hat{\Sigma}_{YX,n}-\beta\hat{\Sigma}_{X,n})\hat{\Sigma}_{X,n}^{-1}(\hat{\Sigma}_{YX,n}-\beta\hat{\Sigma}_{X,n})^\top\quad\text{on }\Omega_n.
\end{align*}
Therefore, Lemma \ref{lemma:ogihara} yields
\[
\|\hat{\Sigma}_{Z,n}-\breve{\Sigma}_{Z,n}\|_{\ell_\infty}\leq r\opnorm{\hat{\Sigma}_{X,n}^{-1}}_2\|\hat{\Sigma}_{YX,n}-\beta\hat{\Sigma}_{X,n}\|_{\ell_\infty}^2\quad\text{on }\Omega_n.
\]
Now, by \ref{ass:eigen.f} and Lemma \ref{inv-sigma-f} we have $\opnorm{\hat{\Sigma}_{X,n}^{-1}}_21_{\Omega_n}=O_p(1)$, so we obtain $\lambda_n^{-2}\|\hat{\Sigma}_{Z,n}-\breve{\Sigma}_{Z,n}\|_{\ell_\infty}1_{\Omega_n}=o_p(r)$ by \ref{f-ass:est}. Since $P(\Omega_n)\to1$ by Lemma \ref{inv-sigma-f}\ref{omega.n}, we complete the proof. 
\end{proof}

\begin{lemma}\label{lemma:sigma.z}
Under the assumptions of Proposition \ref{factor:rate}, $\lambda_n^{-1}\|\hat{\Sigma}_{Z,n}-\Sigma_Z\|_{\ell_\infty}\to^p0$ and $P(\min_{1\leq i\leq d}\hat{\Sigma}_{Z,n}^{ii}>0)\to1$ as $n\to\infty$.
\end{lemma}

\begin{proof}
The first claim immediately follows from Lemma \ref{lemma:approx-z} and \ref{f-ass:est}. The second one is a consequence of the first one, Lemma \ref{lemma:poincare} and \ref{ass:eigen.z}.  
\end{proof}

\begin{proof}[Proof of Proposition \ref{factor:rate}]
Set $\mcl{E}_n:=\Omega_n\cap\{\ol{\Sigma}_n\in\mcl{S}_d^+\}\cap\{\min_{1\leq i\leq d}\hat{\Sigma}_{Z,n}^{ii}>0\}$. From Eq.(0.8.5.3) in \cite{HJ2013}, we have $\hat{\Sigma}_{Z,n}\in\mcl{S}_d^+$ on $\mcl{E}_n$. Hence, from the proof of \cite[Lemma 1]{DGK2008}, the optimization problem in \eqref{f-wglasso} has the unique solution on $\mcl{E}_n$. Since $P(\mcl{E}_n)\to1$ as $n\to\infty$ by \ref{f-ass:psd} and Lemmas \ref{inv-sigma-f} and \ref{lemma:sigma.z}, the desired result follows once we prove $\lambda_n^{-1}\|\hat{\Sigma}_{Z,n}-\Sigma_Z\|_{\ell_\infty}\to^p0$ as $n\to\infty$ according to Proposition \ref{glasso:rate}. This has already been established in Lemma \ref{lemma:sigma.z}. 
\end{proof}

\subsection{Proof of Proposition \ref{factor:max-rate}}

We first establish some asymptotic properties of $\hat{\beta}_n$ which are necessary for the subsequent proofs. 
\begin{lemma}\label{lemma:factor-cov}
Under the assumptions of Proposition \ref{factor:rate}, we have the following results:
\begin{enumerate}[label=(\alph*)]


\item\label{beta-norm} $\|\beta\|_{\ell_2}=O_p(\sqrt{d})$ as $n\to\infty$. 

\item\label{lemma:beta} $\lambda_n^{-1}\max_{1\leq i\leq d}\|\hat{\beta}_n^{i\cdot}-\beta^{i\cdot}\|_{\ell_2}=o_p(\sqrt{r})$ and $\max_{1\leq i\leq d}\|\hat{\beta}_n^{i\cdot}\|_{\ell_2}=O_p(\sqrt{r})$ as $n\to\infty$. 

\item\label{ell2-beta} $\lambda_n^{-1}\|\hat{\beta}_n-\beta\|_{\ell_2}=o_p(\sqrt{dr})$ and $\|\hat{\beta}_n\|_{\ell_2}=O_p(\sqrt{d})$ as $n\to\infty$. 

\item\label{op1-beta} $\lambda_n^{-1}\opnorm{\hat{\beta}_n-\beta}_{1}=o_p(d\sqrt{r})$ and $\opnorm{\hat{\beta}_n}_{1}=O_p(d)$ as $n\to\infty$. 

\item\label{opi-beta} $\lambda_n^{-1}\opnorm{\hat{\beta}_n-\beta}_{\infty}=o_p(r)$ and $\opnorm{\hat{\beta}_n}_{\infty}=O_p(r)$ as $n\to\infty$. 


\end{enumerate}
\end{lemma}

\begin{proof}
\ref{beta-norm} 
Since $\Sigma_X-\Lambda_{\min}(\Sigma_X)\mathsf{E}_r$ is positive semidefinite, $\beta\Sigma_X\beta^\top-\Lambda_{\min}(\Sigma_X)\beta\beta^\top$ is also positive semidefinite. Thus $\Sigma_Y-\Lambda_{\min}(\Sigma_X)\beta\beta^\top$ is positive definite by \eqref{eq:factor-qc}. This implies that $0\leq\trace(\Sigma_Y-\Lambda_{\min}(\Sigma_X)\beta\beta^\top)=\trace(\Sigma_Y)-\Lambda_{\min}(\Sigma_X)\|\beta\|_{\ell_2}^2$. Since $\trace(\Sigma_Y)=O_p(d)$ by \ref{ass:vol}, we obtain $\|\beta\|_{\ell_2}^2=O_p(d)$ by \ref{ass:eigen.f}.

\ref{lemma:beta} 
By Lemma \ref{inv-sigma-f}, on the event $\Omega_n$, we have $\hat{\beta}_n=\hat{\Sigma}_{YX,n}\hat{\Sigma}_{X,n}^{-1}$. Hence, for every $i=1,\dots,d$,
\begin{align*}
\|\hat{\beta}_n^{i\cdot}-\beta^{i\cdot}\|_{\ell_2}&=\|(\hat{\Sigma}_{YX,n}^{i\cdot}-\beta\hat{\Sigma}_{X,n}^{i\cdot})\hat{\Sigma}_{X,n}^{-1}\|_{\ell_2}
\leq\opnorm{\hat{\Sigma}_{X,n}^{-1}}_2\|\hat{\Sigma}_{YX,n}^{i\cdot}-\beta\hat{\Sigma}_{X,n}^{i\cdot}\|_{\ell_2}\\
&\leq\sqrt{r}\opnorm{\hat{\Sigma}_{X,n}^{-1}}_2\|\hat{\Sigma}_{YX,n}-\beta\hat{\Sigma}_{X,n}\|_{\ell_\infty} \qquad\text{on }\Omega_n.
\end{align*}
Since $\opnorm{\hat{\Sigma}_{X,n}^{-1}}_21_{\Omega_n}=O_p(1)$ by Lemma \ref{inv-sigma-f}, $\lambda_n^{-1}\max_{1\leq i\leq d}\|\hat{\beta}_n^{i\cdot}-\beta^{i\cdot}\|_{\ell_\infty}1_{\Omega_n}=o_p(\sqrt{r})$ by \ref{f-ass:est}. Since $P(\Omega_n^c)\to0$ by Lemma \ref{inv-sigma-f}, we obtain $\lambda_n^{-1}\max_{1\leq i\leq d}\|\hat{\beta}_n^{i\cdot}-\beta^{i\cdot}\|_{\ell_2}=o_p(\sqrt{r})$. Since $\max_{1\leq i\leq d}\|\beta^{i\cdot}\|_{\ell_2}\leq \sqrt{r}\|\beta\|_{\ell_\infty}=O(\sqrt{r})$ by \ref{ass:vol}, we also obtain $\max_{1\leq i\leq d}\|\hat{\beta}_n^{i\cdot}\|_{\ell_\infty}=O_p(\sqrt{r})$. 

\ref{ell2-beta} This follows from \ref{beta-norm}--\ref{lemma:beta} and $r\lambda_n=o_p(1)$. 

\ref{op1-beta} This is a direct consequence of \ref{lemma:beta}. 

\ref{opi-beta} This follows from \ref{lemma:beta} and the Schwarz inequality. 
\end{proof}

\begin{proof}[Proof of Proposition \ref{factor:max-rate}]
Since $\|A\|_{\ell_\infty}\leq\opnorm{A}_2$ for any matrix $A$, in view of Proposition \ref{factor:rate} it suffices to prove $\lambda_n^{-1}\|\hat{\beta}_n\hat{\Sigma}_{X,n}\hat{\beta}_n^\top-\beta\Sigma_X\beta^\top\|_{\ell_\infty}=O_p(r^2)$. By Lemma \ref{lemma:ogihara} we have
\begin{align*}
\|\hat{\beta}_n\hat{\Sigma}_{X,n}\hat{\beta}_n^\top-\beta\Sigma_X\beta^\top\|_{\ell_\infty}
&\leq\opnorm{\hat{\Sigma}_{X,n}}_2\lpa\max_{1\leq i\leq d}\|\hat{\beta}_n^{i\cdot}-\beta^{i\cdot}\|_{\ell_2}\rpa\lpa\max_{1\leq i\leq d}\|\hat{\beta}_n^{i\cdot}\|_{\ell_2}\rpa\\
&\quad+\opnorm{\hat{\Sigma}_{X,n}-\Sigma_X}_2\lpa\max_{1\leq i\leq d}\|\beta^{i\cdot}\|_{\ell_2}\rpa\lpa\max_{1\leq i\leq d}\|\hat{\beta}_n^{i\cdot}\|_{\ell_2}\rpa\\
&\quad+\opnorm{\Sigma_X}_2\lpa\max_{1\leq i\leq d}\|\beta^{i\cdot}\|_{\ell_2}\rpa\lpa\max_{1\leq i\leq d}\|\hat{\beta}_n^{i\cdot}-\beta^{i\cdot}\|_{\ell_2}\rpa.
\end{align*}
Therefore, the desired result follows from Lemmas \ref{inv-sigma-f}, \ref{lemma:factor-cov}\ref{lemma:beta} and assumption. 
\end{proof}

\subsection{Proof of Proposition \ref{factor:inverse}}\label{proof:factor:inverse}

Set $\Pi:=(\Sigma_{X}^{-1}+\beta^\top\Sigma_{Z}^{-1}\beta)^{-1}$ and $\hat{\Pi}_n:=(\hat{\Sigma}_{X,n}^{\dagger}+\hat{\beta}_n^\top\hat{\Theta}_{Z,\lambda_n}\hat{\beta}_n)^{-1}$. 
\begin{lemma}\label{lemma:loading}
Under the assumptions of Proposition \ref{factor:inverse}, we have the following results:
\begin{enumerate}[label=(\alph*)]

\item\label{sp-beta} $\Lambda_{\min}(\beta^\top\beta)^{-1}=O(d^{-1})$ as $n\to\infty$. 

\item\label{sp-factor} $\opnorm{\Pi}_2=O_p(d^{-1})$ as $n\to\infty$. 

\item\label{sp-pihat} $\lambda_n^{-1}\opnorm{\hat{\Pi}_n-\Pi}_2=O_p(d^{-1}(s_n+r))$ and $\opnorm{\hat{\Pi}_n}_2=O_p(d^{-1})$ as $n\to\infty$. 

\end{enumerate}
\end{lemma}

\begin{proof}
\ref{sp-beta} 
By Lemma \ref{lemma:weyl} we have $|\Lambda_{\min}(d^{-1}\beta^\top\beta)-\Lambda_{\min}(\bs{B})|\leq\opnorm{d^{-1}\beta^\top\beta-\bs{B}}_2$. Hence the desired result follows from \ref{ass:loading}. 

\ref{sp-factor} 
Since $\opnorm{\Pi}_2=\Lambda_{\min}(\Sigma_{X}^{-1}+\beta^\top\Sigma_{Z}^{-1}\beta)^{-1}$ and $\Sigma_X^{-1}$ is positive definite, Corollary 4.3.12 in \cite{HJ2013} and Lemma \ref{lemma:rayleigh} yield
\begin{align*}
\opnorm{\Pi}_2\leq\Lambda_{\min}(\beta^\top\Sigma_{Z}^{-1}\beta)^{-1}
\leq\Lambda_{\min}(\beta^\top\beta)^{-1}\Lambda_{\min}(\Sigma_{Z}^{-1})^{-1}
=\Lambda_{\min}(\beta^\top\beta)^{-1}\Lambda_{\max}(\Sigma_Z).
\end{align*}
Thus, the desired result follows from claim \ref{sp-beta} and \ref{ass:eigen.z}. 

\ref{sp-pihat} 
First, since we have
\begin{align*}
&\opnorm{\hat{\beta}_n^\top\hat{\Theta}_{Z,\lambda_n}\hat{\beta}_n-\beta^\top\Theta_Z\beta}_2\\
&\leq\opnorm{\hat{\beta}_n-\beta}_2\opnorm{\hat{\Theta}_{Z,\lambda_n}}_2\opnorm{\hat{\beta}_n}_2
+\opnorm{\beta}_2\opnorm{\hat{\Theta}_{Z,\lambda_n}-\Theta_Z}_2\opnorm{\hat{\beta}_n}_2
+\opnorm{\beta}_2\opnorm{\Theta_{Z}}_2\opnorm{\hat{\beta}_n-\beta}_2,
\end{align*}
Lemma \ref{lemma:factor-cov}\ref{beta-norm} and \ref{ell2-beta} and Proposition \ref{factor:rate} yield $\lambda_n^{-1}\opnorm{\hat{\beta}_n^\top\hat{\Theta}_{Z,\lambda_n}\hat{\beta}_n-\beta^\top\Theta_Z\beta}_2=O_p(ds_n)$. Combining this with Lemma \ref{inv-sigma-f} and \ref{sp-factor}, we obtain $\lambda_n^{-1}\opnorm{\Pi(\hat{\Pi}_n^{-1}-\Pi^{-1})}_21_{\Omega_n}=O_p(s_n+r)$. Now let us set $\Omega_{n,1}:=\Omega_n\cap\{\opnorm{\Pi(\hat{\Pi}_n^{-1}-\Pi^{-1})}_2\leq1/2\}$. Then, using \ref{sp-factor} and Lemmas \ref{inv-diff-op} and \ref{inv-sigma-f}\ref{omega.n}, we obtain $\lambda_n^{-1}\opnorm{\hat{\Pi}_n-\Pi}_21_{\Omega_{n,1}}=O_p(d^{-1}(s_n+r))$ and $P(\Omega_{n,1}^c)\to0$. This completes the proof. 
\end{proof}

\begin{proof}[Proof of Proposition \ref{factor:inverse}]
By Sherman-Morisson-Woodbury formula (cf.~Eq.(0.7.4.1) in \cite{HJ2013}), for any $w\in\{2,\infty\}$ we have
\begin{align*}
&\opnorm{\hat{\Sigma}_{Y,\lambda_n}^{-1}-\Sigma_{Y}^{-1}}_w\\
&\leq\opnorm{\hat{\Theta}_{Z,\lambda_n}-\Theta_Z}_w
+\opnorm{(\hat{\Theta}_{Z,\lambda_n}-\Theta_Z)\hat{\beta}_{n}\hat{\Pi}_n\hat{\beta}_{n}^\top\hat{\Theta}_{Z,\lambda_n}}_w
+\opnorm{\Theta_Z(\hat{\beta}_{n}-\beta)\hat{\Pi}_n\hat{\beta}_{n}^\top\hat{\Theta}_{Z,\lambda_n}}_w\\
&\quad+\opnorm{\Theta_Z\beta(\hat{\Pi}_n-\Pi)\hat{\beta}_{n}^\top\hat{\Theta}_{Z,\lambda_n}}_w
+\opnorm{\Theta_Z\beta\Pi(\hat{\beta}_{n}-\beta)^\top\hat{\Theta}_{Z,\lambda_n}}_w
+\opnorm{\Theta_Z\beta\Pi\beta^\top(\hat{\Theta}_{Z,\lambda_n}-\Theta_Z)}_w\\
&=:\Delta_1+\Delta_2+\Delta_3+\Delta_4+\Delta_5+\Delta_6.
\end{align*}
Proposition \ref{factor:rate} yields $\lambda_n^{-1}\Delta_1=O_p(s_n)$. 
Moreover, noting that $\opnorm{\Theta_Z}_\infty=O_p(\sqrt{\mf{d}_n})$ by Lemma \ref{lemma:op-sp} and \ref{ass:eigen.z}, Proposition \ref{factor:rate} and Lemmas \ref{lemma:factor-cov}--\ref{lemma:loading} imply that $\lambda_n^{-1}(\Delta_2+\Delta_6)=O_p(s_n)$, $\lambda_n^{-1}(\Delta_3+\Delta_5)=o_p(\sqrt{r})$ and $\lambda_n^{-1}\Delta_4=O_p(s_n+r)$ when $w=2$ and $\lambda_n^{-1}(\Delta_2+\Delta_6)=O_p(r^{3/2}s_n\sqrt{\mf{d}_n})$, $\lambda_n^{-1}\Delta_3=o_p(r^{3/2}\mf{d}_n)$, $\lambda_n^{-1}\Delta_4=O_p(r^{3/2}(s_n+r)\mf{d}_n)$ and $\lambda_n^{-1}\Delta_5=o_p(r^{2}\mf{d}_n)$ when $w=\infty$. This completes the proof. 
\end{proof}

%
%
%
%
%
%

\subsection{Proof of Proposition \ref{factor:AMN}}

We apply Proposition \ref{prop:AMN} to $\hat{\Sigma}_{Z,n}$. From the arguments in the proof of Proposition \ref{factor:rate}, it remains to check condition \eqref{est:AMN}. More precisely, we need to prove 
\[
\lim_{n\to\infty}\sup_{y\in\mathbb{R}^m}\left|P\left(a_n\wt{J}_{Z,n}\vectorize\left(\hat{\Sigma}_{Z,n}-\Sigma_Z\right)\leq y\right)-P\left(\wt{J}_{Z,n}\mathfrak{C}_n^{1/2}\zeta_n\leq y\right)\right|=0.
\]
Thanks to Lemma 3.1 in \cite{Koike2018sk} and \eqref{f-est:AMN}, this claim follows once we prove $\sqrt{\log (m+1)}\|a_n\wt{J}_{Z,n}\vectorize(\hat{\Sigma}_{Z,n}-\Sigma_Z)-a_n\wt{J}_{Z,n}\vectorize(\breve{\Sigma}_{Z,n}-\Sigma_Z)\|_{\ell_\infty}\to0$. Since we have
\[
\|a_n\wt{J}_{Z,n}\vectorize(\hat{\Sigma}_{Z,n}-\Sigma_Z)-a_n\wt{J}_{Z,n}\vectorize(\breve{\Sigma}_{Z,n}-\Sigma_Z)\|_{\ell_\infty}
\leq a_n\opnorm{J_n}_\infty\opnorm{\Theta_Z}_\infty^2\|\hat{\Sigma}_{Z,n}-\breve{\Sigma}_{Z,n}\|_{\ell_\infty}
\]
and $\opnorm{\Theta_Z}_\infty=O_p(\sqrt{\mf{d}_n})$ by Lemma \ref{lemma:op-sp}, the desired result follows from Lemma \ref{lemma:approx-z} and assumption. \hfill\qed

\subsection{Proof of Lemma \ref{lemma:f-inv-AMN} and Proposition \ref{coro:f-inv-AMN}}

We use the same notation as in Section \ref{proof:factor:inverse}. 
By Sherman-Morisson-Woodbury formula we have
\begin{align*}
\|\hat{\Sigma}_{Y,\lambda_n}^{-1}-\Theta_{Z,\lambda_n}\|_{\ell_\infty}
&\leq\|\hat{\Theta}_{Z,\lambda_n}\hat{\beta}_{n}\hat{\Pi}_n\hat{\beta}_{n}^\top\hat{\Theta}_{Z,\lambda_n}\|_{\ell_\infty}\\
&\leq r\|\hat{\Theta}_{Z,\lambda_n}\hat{\beta}_{n}\|_{\ell_\infty}^2\opnorm{\hat{\Pi}_n}_2
\leq r\opnorm{\hat{\Theta}_{Z,\lambda_n}}_\infty^2\|\hat{\beta}_{n}\|_{\ell_\infty}^2\opnorm{\hat{\Pi}_n}_2,
\end{align*}
where the second inequality follows from Lemma \ref{lemma:ogihara}. Since $\opnorm{\Theta_Z}_\infty^2=O(\mf{d}_n)$ by Lemma \ref{lemma:op-sp}, we have $\opnorm{\hat{\Theta}_{Z,\lambda_n}}_\infty^2=O_p(\mf{d}_n)$ by Proposition \ref{factor:rate}. We also have $\|\hat{\beta}_{n}\|_{\ell_\infty}=O_p(1)$ by \ref{ass:vol}, \ref{f-ass:rate} and Lemma \ref{lemma:factor-cov}\ref{lemma:beta}. Consequently, we obtain $\|\hat{\Sigma}_{Y,\lambda_n}^{-1}-\Theta_{Z,\lambda_n}\|_{\ell_\infty}=O_p(r\mf{d}_n/d)$ by Lemma \ref{lemma:loading}. 
Similarly, we can prove $\|\Sigma_{Y}^{-1}-\Theta_{Z}\|_{\ell_\infty}=O_p(r\mf{d}_n/d)$. So we complete the proof of Lemma \ref{lemma:f-inv-AMN}.  

Proposition \ref{coro:f-inv-AMN} is an immediate consequence of Proposition \ref{factor:AMN}, Lemma \ref{lemma:f-inv-AMN} and \cite[Lemma 3.1]{Koike2018sk}.\hfil\qed

\section{Proofs for Section \ref{sec:rc}}

\subsection{Proof of Theorem \ref{rc:rate}}

The proof relies on the following concentration inequalities for discretized quadratic covariations of continuous martingales: 
\begin{lemma}\label{rc:concentrate}
Let $M=(M_t)_{t\in[0,1]}$ and $N=(N_t)_{t\in[0,1]}$ be two continuous martingales. 
Suppose that there is a constant $L>0$ such that
\begin{equation}\label{eq:lipschitz}
|[M,M]_t-[M, M]_s|\vee|[N, N]_t-[N, N]_s|\leq L|t-s|
\end{equation}
for all $s,t\in[0,1]$. Then, for any $\theta>0$, there is a constant $C_{L,\theta}>0$ which depends only on $L$ and $\theta$ such that
\[
P\left(\sqrt{n}\labs\wh{[M,N]}^n_1-[M,N]_1\rabs>x\right)\leq2\exp\lpa-C_{L,\theta}x^2\rpa
\]
for all $n\in\mathbb{N}$ and $x\in[0,\theta\sqrt{n}]$. 
\end{lemma}
\begin{rmk}\label{rmk:concentrate}
Similar estimates to Lemma \ref{rc:concentrate} have already been obtained in the literature (see e.g.~\cite[Lemma 3]{FLY2012}, \cite[Lemma 10]{TWZ2013} and \cite[Lemma A.1]{FFX2016}). Since we use slightly different assumptions from the existing ones, we give its proof in Appendix \ref{proof:concentrate} for the shake of completeness. 
\end{rmk}
Define the $(d+r)$-dimensional semimartingale $\bar{Z}=(\bar{Z}_t)_{t\in[0,1]}$ by $\bar{Z}_t=(Z_t^1,\dots,Z_t^d,X_t^1,\dots,X_t^r)^\top$. 
\begin{lemma}\label{lemma:rc-max}
Assume \ref{rc-bdd} and $\log (d+r)/\sqrt{n}\to0$ as $n\to\infty$. 
Then, $\|\wh{[\bar{Z},\bar{Z}]}_1^n-[\bar{Z},\bar{Z}]_1\|_{\ell_\infty}=O_p(\sqrt{\log (d+r)/n})$ as $n\to\infty$. 
\end{lemma}

\begin{proof}
For all $n,\nu\in\mathbb{N}$ and $t\in[0,1]$, set
\[
\bar{\mu}(\nu)_t=\begin{pmatrix}
\mu(\nu)_t\\
\wt{\mu}(\nu)_t
\end{pmatrix},\qquad
\bar{\sigma}(\nu)_t=\begin{pmatrix}
\sigma(\nu)_t\\
\wt{\sigma}(\nu)_t
\end{pmatrix}.
\]
Then we define the processes $\bar{A}(\nu)=(\bar{A}(\nu)_t)_{t\in[0,1]}$ and $\bar{M}(\nu)=(\bar{M}(\nu)_t)_{t\in[0,1]}$ by $\bar{A}(\nu)_t=\int_0^t\bar{\mu}(\nu)_sds$ and $\bar{M}(\nu)_t=\int_0^t\bar{\sigma}(\nu)_sdW_s$. By the local property of It\^o integrals (cf.~pages 17--18 of \cite{Nualart2006}), we have $\bar{Z}=\bar{Z}(\nu):=\bar{A}(\nu)+\bar{M}(\nu)$ on $\Omega_n(\nu)$. Hence, for every $L>0$, it holds that
\begin{align*}
&P\lpa\|\wh{[\bar{Z},\bar{Z}]}_1^n-[\bar{Z},\bar{Z}]_1\|_{\ell_\infty}>L\sqrt{\log (d+r)/n}\rpa\\
&\leq P\lpa\|\wh{[\bar{Z}(\nu),\bar{Z}(\nu)]}_1^n-[\bar{Z}(\nu),\bar{Z}(\nu)]_1\|_{\ell_\infty}>L\sqrt{\log (d+r)/n}\rpa
+P(\Omega_n(\nu)^c).
\end{align*}
Therefore, the proof is completed once we show that
\[
\lim_{L\to\infty}\limsup_{n\to\infty}P\lpa\|\wh{[\bar{Z}(\nu),\bar{Z}(\nu)]}_1^n-[\bar{Z}(\nu),\bar{Z}(\nu)]_1\|_{\ell_\infty}>L\sqrt{\log (d+r)/n}\rpa=0
\]
for any fixed $\nu>0$. 
We decompose the target quantity as
\begin{align*}
&\wh{[\bar{Z}(\nu),\bar{Z}(\nu)]}_1^n-[\bar{Z}(\nu),\bar{Z}(\nu)]_1\\
&=(\wh{[\bar{M}(\nu),\bar{M}(\nu)]}_1^n-[\bar{M}(\nu),\bar{M}(\nu)]_1)
+\wh{[\bar{A}(\nu),\bar{A}(\nu)]}_1^n
+\wh{[\bar{A}(\nu),\bar{M}(\nu)]}_1^n
+\wh{[\bar{M}(\nu),\bar{A}(\nu)]}_1^n\\
&=:\mathbb{I}_n+\mathbb{II}_n+\mathbb{III}_n+\mathbb{IV}_n.
\end{align*}
First we consider $\mathbb{I}_n$. Since we have $|[\bar{M}(\nu)^i,\bar{M}(\nu)^i]_t-[\bar{M}(\nu)^i,\bar{M}(\nu)^i]_s|\leq C_\nu|t-s|$ for all $s,t\in[0,1]$ and $i\in\{1,\dots,d+r\}$ by \ref{rc-bdd}, by Lemma \ref{rc:concentrate} there is a constant $C>0$ such that
\[
\max_{1\leq i,j\leq d+r}P\left(\sqrt{n}\left|\mathbb{I}_n^{ij}\right|>x\right)\leq2e^{-Cx^2}
\]
for all $n\in\mathbb{N}$ and $x\in[0,\sqrt{n}]$. Therefore, for every $L\in[0,\sqrt{n/\log (d+r)}]$ we obtain
\begin{align*}
P\left(\left\|\mathbb{I}_n\right\|_{\ell_\infty}>L\sqrt{\frac{\log (d+r)}{n}}\right)
&\leq \sum_{i,j=1}^{d+r}P\left(\sqrt{n}\left|\mathbb{I}_n^{ij}\right|>L\sqrt{\log (d+r)}\right)
\leq 2(d+r)^{2-CL^2}.
\end{align*}
Hence, noting the assumption $\sqrt{n}/\log (d+r)\to\infty$, we conclude that
\begin{align*}
\lim_{L\to\infty}\limsup_{n\to\infty}P\left(\left\|\mathbb{I}_n\right\|_{\ell_\infty}>L\sqrt{\frac{\log (d+r)}{n}}\right)=0.
\end{align*}
Next, by \ref{rc-bdd} we have $\|\mathbb{II}_n\|_{\ell_\infty}\leq C_\nu^2/n$. So we obtain $\|\mathbb{II}_n\|_{\ell_\infty}=O(n^{-1})=O(\sqrt{\log (d+r)/n})$. 
Third, we consider $\mathbb{III}_n$. By the Schwarz inequality we have 
\[
\|\mathbb{III}_n\|_{\ell_\infty}\leq\sqrt{\|\mathbb{II}_n\|_{\ell_\infty}}\max_{1\leq j\leq d+r}\sqrt{\wh{[\bar{M}(\nu)^j,\bar{M}(\nu)^j]}_1^n}.
\] 
From the above result we have $\sqrt{\|\mathbb{II}_n\|_{\ell_\infty}}=O(1/\sqrt{n})$. Meanwhile, using the inequality $\sqrt{x}\leq\sqrt{|x-y|}+\sqrt{y}$ holding for all $x,y\geq0$, we have
\[
\max_{1\leq j\leq d+r}\sqrt{\wh{[\bar{M}(\nu)^j,\bar{M}(\nu)^j]}_1^n}
\leq\sqrt{\|\mathbb{I}_n\|_{\ell_\infty}}+\max_{1\leq j\leq d+r}\sqrt{[\bar{M}(\nu)^j,\bar{M}(\nu)^j]_1}
\leq\sqrt{\|\mathbb{I}_n\|_{\ell_\infty}}+\sqrt{C_\nu}.
\]
Hence the above result yields $\max_{1\leq j\leq d+r}\sqrt{\wh{[\bar{M}(\nu)^j,\bar{M}(\nu)^j]}_1^n}=O_p(1)$. Thus, we conclude that $\|\mathbb{III}_n\|_{\ell_\infty}=O_p(1/\sqrt{n})=O_p(\sqrt{\log (d+r)/n})$. 
Finally, since $\|\mathbb{IV}_n\|_{\ell_\infty}=\|\mathbb{III}_n\|_{\ell_\infty}$, we complete the proof.
\end{proof}

\begin{proof}[Proof of Theorem \ref{rc:rate}]
In view of Propositions \ref{factor:rate}--\ref{factor:inverse}, it suffices to check \ref{f-ass:est}. Noting that $\hat{\Sigma}_{YX,n}-\beta\hat{\Sigma}_{X,n}=\wh{[Z,X]}_1^n$ and $\breve{\Sigma}_{Z,n}=\wh{[Z,Z]}_1^n$, \ref{f-ass:est} immediately follows from Lemma \ref{lemma:rc-max}. 
\end{proof}

\subsection{Proof of Theorem \ref{rc:AMN}}

Our proof relies on the following ``high-dimensional'' asymptotic mixed normality of the realized covariance matrix:
\begin{lemma}[\cite{Koike2018sk}, Theorem 4.2(b)]\label{lemma:rc-AMN}
Assume \ref{rc-mal}. 
For every $n$, let $\bs{X}_n$ be an $m\times d^2$ random matrix and $\Upsilon_n$ be an $m\times d^2$ non-random matrix such that $\opnorm{\Upsilon_n}_\infty\geq1$, where $m=m_n$ possibly depends on $n$. 
Define $\Xi_n:=\Upsilon_n\circ\bs{X}_n$. 
Suppose that, for all $n,\nu\in\mathbb{N}$, we have $\bs{X}_n(\nu)\in\mathbb{D}_{2,\infty}(\mathbb{R}^{m\times d^2})$ such that $\bs{X}_n=\bs{X}_n(\nu)$ on $\Omega_n(\nu)$ and
\begin{align}
&\lim_{b\downarrow0}\limsup_{n\to\infty}P(\min\diag(\Xi_n(\nu)\mathfrak{C}_n(\nu)\Xi_n(\nu)^\top)<b)=0,\label{rc-diag}\\
&\sup_{n\in\mathbb{N}}
\max_{1\leq i\leq m}\max_{1\leq j\leq d^2}\left(
\|\bs{X}_n(\nu)^{ij}\|_\infty
+\sup_{0\leq t\leq 1}\|D_t\bs{X}_n(\nu)^{ij}\|_{\infty,\ell_2}
+\sup_{0\leq s,t\leq 1}\|D_{s,t}\bs{X}_n(\nu)^{ij}\|_{\infty,\ell_2}
\right)
<\infty,\label{eq-X}
\end{align}
where $\Xi_n(\nu):=\Upsilon_n\circ\bs{X}_n(\nu)$. Suppose also $\opnorm{\Upsilon_n}_\infty^5(\log dm)^{\frac{13}{2}}\to0$ as $n\to\infty$. 
Then we have
\begin{equation}\label{rc:result}
\sup_{y\in\mathbb{R}^m}\left|P\left(\Xi_n\vectorize\left(\wh{[Z,Z]}_1^n-[Z,Z]_1\right)\leq y\right)-P(\Xi_n\mathfrak{C}_n^{1/2}\zeta_n\leq y)\right|\to0
\end{equation}
as $n\to\infty$.
\end{lemma}

To apply Lemma \ref{lemma:rc-AMN} to the present setting, we prove some auxiliary results. 
\begin{lemma}\label{kronecker-hadamard}
Let $A_1,A_2,B_1,B_2\in\mathbb{R}^{d\times d}$. Then $(A_1\otimes A_2)\circ(B_1\otimes B_2)=(A_1\circ B_1)\otimes(A_2\circ B_2)$.  
\end{lemma}

\begin{proof}
This follows from a straightforward computation. 
%
\end{proof}

\begin{lemma}\label{S-deriv}
Assume \ref{rc-mal}. Then, for any $n,\nu\in\mathbb{N}$ and $t\in[0,1]$, $c(\nu)_t\in\mathbb{D}_{2,\infty}(\mathbb{R}^{d\times d})$, $\mf{C}_n(\nu)\in\mathbb{D}_{2,\infty}(\mathbb{R}^{d^2\times d^2})$ and
\begin{align*}
&\sup_{n\in\mathbb{N}}\max_{1\leq i,j\leq d}\left(
\sup_{0\leq t,u\leq 1}\|D_sc(\nu)_t^{ij}\|_{\infty,\ell_2}
+\sup_{0\leq t,u,v\leq 1}\|D_{u,v}c(\nu)_t^{ij}\|_{\infty,\ell_2}
\right)<\infty,\\
&\sup_{n\in\mathbb{N}}\max_{1\leq k,l\leq d^2}\left(
\sup_{0\leq u\leq 1}\|D_u\mf{C}_n(\nu)^{kl}\|_{\infty,\ell_2}
+\sup_{0\leq u,v\leq 1}\|D_{u,v}\mf{C}_n(\nu)^{kl}\|_{\infty,\ell_2}
\right)<\infty.
\end{align*}
\end{lemma}
\begin{proof}
This directly follows from Lemmas B.11--B.12 in \cite{Koike2018sk}. 
\end{proof}

\begin{lemma}\label{Theta-deriv}
Assume \ref{rc-mal}--\ref{ass:adjacent}. For any $n,\nu\in\mathbb{N}$, $\Theta_Z(\nu)\in\mathbb{D}_{2,\infty}(\mathbb{R}^{d\times d})$ and
\[
\sup_{n\in\mathbb{N}}\max_{1\leq i,j\leq d}\left(
\sup_{0\leq t\leq 1}\|D_t\Theta_Z(\nu)^{ij}\|_{\infty,\ell_2}
+\sup_{0\leq s,t\leq 1}\|D_{s,t}\Theta_Z(\nu)^{ij}\|_{\infty,\ell_2}
\right)<\infty.
\]
\end{lemma}

\begin{proof}
First, by Remark 15.87 in \cite{Janson1997} and Lemma \ref{S-deriv}, $\Sigma_Z(\nu)\in\mathbb{D}_{2,\infty}(\mathbb{R}^{d\times d})$ and $D^k\Sigma_Z(\nu)=\int_0^1D^kc(\nu)_sds$ for $k=1,2$. In particular, we have
\begin{equation}\label{eq-IC}
\sup_{n\in\mathbb{N}}\max_{1\leq i,j\leq d}\left(
\sup_{0\leq t\leq 1}\|D_t\Sigma_Z(\nu)^{ij}\|_{\infty,\ell_2}
+\sup_{0\leq s,t\leq 1}\|D_{s,t}\Sigma_Z(\nu)^{ij}\|_{\infty,\ell_2}
\right)<\infty
\end{equation}
by Lemma \ref{S-deriv} and \eqref{eq-sigma}. 
Next, by Theorem 15.78 in \cite{Janson1997} and Theorem 4 in \cite[Chapter 8]{MN1988}, we have $\Theta_Z(\nu)\in\mathbb{D}_{2,\infty}(\mathbb{R}^{d\times d})$ with 
$
D_s^{(a)}\Theta_Z(\nu)=-\Theta_Z(\nu)D_s^{(a)}\Sigma_Z(\nu)\Theta_Z(\nu)
$ 
and
\begin{multline*}
D_{s,t}^{(a,b)}\Theta_Z(\nu)=\Theta_Z(\nu)D_t^{(b)}\Sigma_Z(\nu)\Theta_Z(\nu)D_s^{(a)}\Sigma_Z(\nu)\Theta_Z(\nu)\\
-\Theta_Z(\nu)D_{s,t}^{(a,b)}\Sigma_Z(\nu)\Theta_Z(\nu)
+\Theta_Z(\nu)D_s^{(a)}\Sigma_Z(\nu)\Theta_Z(\nu)D_t^{(b)}\Sigma_Z(\nu)\Theta_Z(\nu)
\end{multline*}
for all $s,t\in[0,1]$ and $a,b\in\{1,\dots,d'\}$. 
Therefore, by Lemma \ref{lemma:schwarz} we have
\[
\|D_s\Theta_Z(\nu)^{ij}\|_{\ell_2}\leq\opnorm{\Theta_Z(\nu)}_\infty^2\max_{1\leq k,l\leq d}\|D_s\Sigma_Z(\nu)^{kl}\|_{\ell_2}
\]
for all $i,j=1,\dots,d$. Then, noting that $Q_{Z}$ is non-random, we have $(1_{\{\Theta_Z(\nu)^{ij}\neq0\}})_{1\leq i,j\leq d}=Q_Z$ by assumption. So we obtain $\opnorm{\Theta_Z(\nu)}_\infty\leq\opnorm{Q_Z}_\infty\|\Theta_Z(\nu)\|_{\ell_\infty}$. Hence, \eqref{eq-IC}, \ref{ass:adjacent} and \eqref{Theta-bdd} yield $\sup_{n\in\mathbb{N}}\max_{1\leq i,j\leq d}\sup_{0\leq t\leq 1}\|D_t\Theta_Z(\nu)^{ij}\|_{\infty,\ell_2}<\infty$. 
In the meantime, by Lemma \ref{lemma:schwarz} we also have
\begin{align*}
&\|D_{s,t}\Theta_Z(\nu)^{ij}\|_{\ell_2}\\
&\leq2\sqrt{\sum_{b=1}^{d'}\lpa\sum_{k=1}^d\labs\lpa\Theta_Z(\nu)D_t^{(b)}\Sigma_Z(\nu)\Theta_Z(\nu)\rpa^{ik}\rabs\rpa^2}\opnorm{\Theta_Z(\nu)}_\infty\max_{1\leq k,l\leq d}\|D_s\Sigma_Z(\nu)^{kl}\|_{\ell_2}\\
&\quad+\opnorm{\Theta_Z(\nu)}_\infty^2\max_{1\leq k,l\leq d}\|D_{s,t}\Sigma_Z(\nu)^{kl}\|_{\ell_2}\\
&=2\sqrt{\sum_{b=1}^{d'}\lpa\sum_{k=1}^d\labs D_t^{(b)}\Theta_Z(\nu)^{ik}\rabs\rpa^2}\opnorm{\Theta_Z(\nu)}_\infty\max_{1\leq k,l\leq d}\|D_s\Sigma_Z(\nu)^{kl}\|_{\ell_2}\\
&\quad+\opnorm{\Theta_Z(\nu)}_\infty^2\max_{1\leq k,l\leq d}\|D_{s,t}\Sigma_Z(\nu)^{kl}\|_{\ell_2}.
\end{align*}
Now, since $Q_{Z}$ is non-random, we have $D_t^{(b)}\Theta_Z(\nu)=Q_Z\circ D_t^{(b)}\Theta_Z(\nu)$. So, the Schwarz inequality yields
\begin{align*}
&\|D_{s,t}\Theta_Z(\nu)^{ij}\|_{\ell_2}\\
&\leq2\sqrt{\opnorm{Q_Z}_\infty\sum_{k=1}^d\sum_{b=1}^{d'}\labs D_t^{(b)}\Theta_Z(\nu)^{ik}\rabs^2}\opnorm{\Theta_Z(\nu)}_\infty\max_{1\leq k,l\leq d}\|D_s\Sigma_Z(\nu)^{kl}\|_{\ell_2}\\
&\quad+\opnorm{\Theta_Z(\nu)}_\infty^2\max_{1\leq k,l\leq d}\|D_{s,t}\Sigma_Z(\nu)^{kl}\|_{\ell_2}\\
&\leq2\opnorm{Q_Z}_\infty\lpa\max_{1\leq k,l\leq d}\|D_t\Theta_Z(\nu)^{kl}\|_{\ell_2}\rpa\opnorm{\Theta_Z(\nu)}_\infty\max_{1\leq k,l\leq d}\|D_s\Sigma_Z(\nu)^{kl}\|_{\ell_2}\\
&\quad+\opnorm{\Theta_Z(\nu)}_\infty^2\max_{1\leq k,l\leq d}\|D_{s,t}\Sigma_Z(\nu)^{kl}\|_{\ell_2}.
\end{align*}
Consequently, we conclude $\sup_{n\in\mathbb{N}}\max_{1\leq i,j\leq d}\sup_{0\leq s,t\leq 1}\|D_{s,t}\Theta_Z(\nu)^{ij}\|_{\infty,\ell_2}<\infty$ by \ref{ass:adjacent}, \eqref{Theta-bdd} and the results proved above.
\end{proof}

\begin{lemma}\label{acov-deriv}
Assume \ref{rc-mal}--\ref{ass:adjacent}. For any $n,\nu\in\mathbb{N}$, $\Theta_Z(\nu)\otimes\Theta_Z(\nu),\mf{V}_n(\nu)\in\mathbb{D}_{2,\infty}(\mathbb{R}^{d^2\times d^2})$ and
\begin{align}
&\sup_{n\in\mathbb{N}}\max_{1\leq i,j\leq d^2}\left(
\sup_{0\leq t\leq 1}\|D_t\{\Theta_Z(\nu)\otimes\Theta_Z(\nu)\}^{ij}\|_{\infty,\ell_2}
+\sup_{0\leq s,t\leq 1}\|D_{s,t}\{\Theta_Z(\nu)\otimes\Theta_Z(\nu)\}^{ij}\|_{\infty,\ell_2}
\right)<\infty,\label{eq-Theta2}\\
&\sup_{n\in\mathbb{N}}\max_{1\leq i,j\leq d^2}\left(
\sup_{0\leq t\leq 1}\|D_t\mf{V}_n(\nu)^{ij}\|_{\infty,\ell_2}
+\sup_{0\leq s,t\leq 1}\|D_{s,t}\mf{V}_n(\nu)^{ij}\|_{\infty,\ell_2}
\right)<\infty.\label{eq-V}
\end{align}
\end{lemma}

\begin{proof}
First, Corollary 15.80 in \cite{Janson1997}, \eqref{Theta-bdd} and Lemma \ref{Theta-deriv} imply that $\Theta_Z(\nu)\otimes\Theta_Z(\nu)\in\mathbb{D}_{2,\infty}(\mathbb{R}^{2d^2\times d^2})$ and \eqref{eq-Theta2} holds true. 
Next, Corollary 15.80 in \cite{Janson1997} and Lemma \ref{S-deriv} imply that $\mf{V}_n(\nu)\in\mathbb{D}_{2,\infty}(\mathbb{R}^{d^2\times d^2})$ and
\begin{align*}
D_s^{(a)}\mf{V}_n(\nu)&=\{D_s^{(a)}\Theta_Z^{\otimes2}(\nu)\}\mf{C}_n(\nu)\Theta_Z^{\otimes2}(\nu)\\
&\quad+\Theta_Z^{\otimes2}(\nu)\{D_s^{(a)}\mf{C}_n(\nu)\}\Theta_Z^{\otimes2}(\nu)
+\Theta_Z^{\otimes2}(\nu)\mf{C}_n(\nu)\{D_s^{(a)}\Theta_Z^{\otimes2}(\nu)\}
\end{align*}
and
\begin{align*}
D_{s,t}^{(a,b)}\mf{V}_n(\nu)&=\{D_{s,t}^{(a,b)}\Theta_Z^{\otimes2}(\nu)\}\mf{C}_n(\nu)\Theta_Z^{\otimes2}(\nu)
+\{D_s^{(a)}\Theta_Z^{\otimes2}(\nu)\}\{D_t^{(b)}\mf{C}_n(\nu)\}\Theta_Z^{\otimes2}(\nu)\\
&\quad+\{D_s^{(a)}\Theta_Z^{\otimes2}(\nu)\}\mf{C}_n(\nu)\{D_t^{(b)}\Theta_Z^{\otimes2}(\nu)\}
+\{D_s^{(b)}\Theta_Z^{\otimes2}(\nu)\}\{D_s^{(a)}\mf{C}_n(\nu)\}\Theta_Z^{\otimes2}(\nu)\\
&\quad+\Theta_Z^{\otimes2}(\nu)\{D_{s,t}^{(a,b)}\mf{C}_n(\nu)\}\Theta_Z^{\otimes2}(\nu)
+\Theta_Z^{\otimes2}(\nu)\{D_s^{(a)}\mf{C}_n(\nu)\}\{D_t^{(b)}\Theta_Z^{\otimes2}(\nu)\}\\
&\quad+\{D_t^{(b)}\Theta_Z^{\otimes2}(\nu)\}\mf{C}_n(\nu)\{D_s^{(a)}\Theta_Z^{\otimes2}(\nu)\}
+\Theta_Z^{\otimes2}(\nu)\{D_t^{(b)}\mf{C}_n(\nu)\}\{D_s^{(a)}\Theta_Z^{\otimes2}(\nu)\}\\
&\quad+\Theta_Z^{\otimes2}(\nu)\mf{C}_n(\nu)\{D_{s,t}^{(a,b)}\Theta_Z^{\otimes2}(\nu)\}
\end{align*}
for any $s,t\in[0,1]$, where $\Theta_Z^{\otimes2}(\nu):=\Theta_Z(\nu)\otimes\Theta_Z(\nu)$. 
Thus, by Lemma \ref{lemma:schwarz} we obtain
\begin{align*}
\max_{1\leq i,j\leq d^2}\|D_s\mf{V}_n(\nu)^{ij}\|_{\ell_2}
&\leq2\max_{1\leq i\leq d^2}\sqrt{\sum_{a=1}^{d'}\lpa\sum_{k=1}^{d^2}\labs D_s^{(a)}\Theta_Z^{\otimes2}(\nu)^{ik}\rabs\rpa^2}\|\mf{C}_n(\nu)\|_{\ell_\infty}\opnorm{\Theta_Z^{\otimes2}(\nu)}_\infty\\
&\quad+\opnorm{\Theta_Z^{\otimes2}(\nu)}_\infty^2\max_{1\leq i,j\leq d^2}\|D_s\mf{C}_n(\nu)^{ij}\|_{\ell_2}
\end{align*}
and
\begin{align*}
&\max_{1\leq i,j\leq d^2}\|D_{s,t}\mf{V}_n(\nu)^{ij}\|_{\ell_2}\\
&\leq2\max_{1\leq i\leq d^2}\sup_{0\leq s,t\leq 1}\sqrt{\sum_{a,b=1}^{d'}\lpa\sum_{k=1}^{d^2}\labs D_{s,t}^{(a,b)}\Theta_Z^{\otimes2}(\nu)^{ik}\rabs\rpa^2}\|\mf{C}_n(\nu)\|_{\ell_\infty}\opnorm{\Theta_Z^{\otimes2}(\nu)}_\infty\\
&\quad+4\max_{1\leq i,j,l\leq d^2}\sup_{0\leq s,t\leq 1}\sqrt{\sum_{a=1}^{d'}\lpa\sum_{k=1}^{d^2}\labs D_{s}^{(a)}\Theta_Z^{\otimes2}(\nu)^{ik}\rabs\rpa^2}\|D_t\mf{C}_n(\nu)^{jl}\|_{\ell_2}\opnorm{\Theta_Z^{\otimes2}(\nu)}_\infty\\
&\quad+2\max_{1\leq i\leq d^2}\sup_{0\leq s\leq 1}\|\mf{C}_n(\nu)\|_{\ell_\infty}\sum_{a=1}^{d'}\lpa\sum_{k=1}^{d^2}\labs D_{s}^{(a)}\Theta_Z^{\otimes2}(\nu)^{ik}\rabs\rpa^2\\
&\quad+\max_{1\leq i,j\leq d^2}\sup_{0\leq s,t\leq 1}\opnorm{\Theta_Z^{\otimes2}(\nu)}_\infty^2\|D_{s,t}\mf{C}_n(\nu)^{ij}\|_{\ell_2}.
\end{align*}
Now, as pointed out in the proof of Lemma \ref{Theta-deriv}, we have $\Theta_Z(\nu)=Q_Z\circ\Theta_Z(\nu)$. So Lemma \ref{kronecker-hadamard} yields $\Theta_Z^{\otimes2}(\nu)=(Q_Z\otimes Q_Z)\circ\Theta_Z^{\otimes2}(\nu)$. Since $Q_Z$ is non-random by \ref{ass:adjacent}, we have $D_s^{(a)}\Theta_Z^{\otimes2}(\nu)=(Q_Z\otimes Q_Z)\circ D_s^{(a)}\Theta_Z^{\otimes2}(\nu)$. Thus, using the Schwarz inequality repeatedly, we obtain
\begin{align*}
\max_{1\leq i,j\leq d^2}\|D_s\mf{V}_n(\nu)^{ij}\|_{\ell_2}
&\leq2\max_{1\leq ij\leq d^2}\opnorm{Q_Z}_\infty^2\|D_s\Theta_Z^{\otimes2}(\nu)^{ij}\|_{\ell_2}\|\mf{C}_n(\nu)\|_{\ell_\infty}\opnorm{\Theta_Z^{\otimes2}(\nu)}_\infty\\
&\quad+\opnorm{\Theta_Z^{\otimes2}(\nu)}_\infty^2\max_{1\leq i,j\leq d^2}\|D_s\mf{C}_n(\nu)^{ij}\|_{\ell_2}
\end{align*}
and
\begin{align*}
&\max_{1\leq i,j\leq d^2}\|D_{s,t}\mf{V}_n(\nu)^{ij}\|_{\ell_2}\\
&\leq2\max_{1\leq i,j\leq d^2}\sup_{0\leq s,t\leq 1}\opnorm{Q_Z}_\infty^2\|D_{s,t}\Theta_Z^{\otimes2}(\nu)^{ij}\|_{\ell_2}\|\mf{C}_n(\nu)\|_{\ell_\infty}\opnorm{\Theta_Z^{\otimes2}(\nu)}_\infty\\
&\quad+4\max_{1\leq i,j,k,l\leq d^2}\sup_{0\leq s,t\leq 1}\opnorm{Q_Z}_\infty^2\|D_{s}\Theta_Z^{\otimes2}(\nu)^{ik}\|_{\ell_2}\|D_t\mf{C}_n(\nu)^{jl}\|_{\ell_2}\opnorm{\Theta_Z^{\otimes2}(\nu)}_\infty\\
&\quad+2\max_{1\leq i,j\leq d^2}\sup_{0\leq s\leq 1}\|\mf{C}_n(\nu)\|_{\ell_\infty}\opnorm{Q_Z}_\infty^2\|D_{s}\Theta_Z^{\otimes2}(\nu)^{ij}\|_{\ell_2}\\
&\quad+\max_{1\leq i,j\leq d^2}\sup_{0\leq s,t\leq 1}\opnorm{\Theta_Z^{\otimes2}(\nu)}_\infty^2\|D_{s,t}\mf{C}_n(\nu)^{ij}\|_{\ell_2}.
\end{align*}
Hence we complete the proof by Lemma \ref{S-deriv}, \eqref{eq-Theta2} and assumption. 
\end{proof}

\begin{proof}[Proof of Theorem \ref{rc:AMN}]
Set $U_n:=\sqrt{n}\vectorize(\hat{\Theta}_{Z,\lambda_n}-\Gamma_{Z,n}-\Theta_Z)$. 
Define the $2d^2\times d^2$ matrices $J_{n,1}$ and $J_{n,2}$ by
\[
J_{n,1}=\begin{pmatrix}
\mathsf{E}_{d^2}\\
-\mathsf{E}_{d^2}
\end{pmatrix},\qquad
J_{n,2}=\begin{pmatrix}
\mf{S}_n^{-1}\\
-\mf{S}_n^{-1}
\end{pmatrix}.
\]
Then we have
\[
\sup_{A\in\mathcal{A}^\mathrm{re}(d^2)}\left|P\left(U_n\in A\right)-P\left(\mf{V}_n^{1/2}\zeta_n\in A\right)\right|
=\sup_{y\in\mathbb{R}^{2d^2}}\left|P\left(J_{n,1}U_n\leq y\right)-P\left(J_{n,1}\mf{V}_n^{1/2}\zeta_n\leq y\right)\right|
\]
and
\[
\sup_{A\in\mathcal{A}^\mathrm{re}(d^2)}\left|P\left(\mf{S}_n^{-1}U_n\in A\right)-P\left(\mf{S}_n^{-1}\mf{V}_n^{1/2}\zeta_n\in A\right)\right|
=\sup_{y\in\mathbb{R}^{2d^2}}\left|P\left(J_{n,2}U_n\leq y\right)-P\left(J_{n,2}\mf{V}_n^{1/2}\zeta_n\leq y\right)\right|.
\]
Therefore, in view of Proposition \ref{factor:AMN}, it suffices to check \ref{f-ass:est} and \eqref{f-est:AMN}--\eqref{factor:diag-tight} for $J_n\in\{J_{n,1},J_{n,2}\}$. 
We have already checked \ref{f-ass:est} in the proof of Theorem \ref{rc:rate}. 
Meanwhile, \eqref{factor:diag-tight} immediately follows from \ref{rc-bdd} and \eqref{Theta-bdd}. 
To check \eqref{f-est:AMN}, we apply Lemma \ref{lemma:rc-AMN} with $\Xi_n=-J_n(\Theta_Z\otimes\Theta_Z)$ (note that $\breve{\Sigma}_{Z,n}=\wh{[Z,Z]}_1^n$). Set
\[
\Upsilon_{n}=\begin{pmatrix}
Q_{Z}\otimes Q_{Z}\\
Q_{Z}\otimes Q_{Z}
\end{pmatrix}.
\]
Then we have $\Xi_n=\Upsilon_n\circ\Xi_n$ by Lemma \ref{kronecker-hadamard}. Since $\Upsilon_n$ is non-random by \ref{ass:adjacent}, we can apply Lemma \ref{lemma:rc-AMN} with $\bs{X}_n=\Xi_n$ once we show that, for every $\nu\in\mathbb{N}$, there is an $\bs{X}_n(\nu)\in\mathbb{D}_{2,\infty}(\mathbb{R}^{m\times d^2})$ such that $\bs{X}_n=\bs{X}_n(\nu)$ on $\Omega_n(\nu)$ and \eqref{rc-diag}--\eqref{eq-X} hold true. 
Now we separately consider the two cases. 

\ul{Case 1: $J_n=J_{n,1}$}. 
In this case we set $\bs{X}_n(\nu):=-J_{n,1}(\Theta_Z(\nu)\otimes\Theta_Z(\nu))$. By \ref{rc-bdd} we have $\bs{X}_n=\bs{X}_n(\nu)$ on $\Omega_n(\nu)$, while \eqref{rc-diag}--\eqref{eq-X} follow from \eqref{Theta-bdd} and Lemma \ref{acov-deriv}, respectively. 

\ul{Case 2: $J_n=J_{n,2}$}. 
In this case we set $\bs{X}_n(\nu):=-J_{n,2}(\nu)(\Theta_Z(\nu)\otimes\Theta_Z(\nu))$, where
\[
J_{n,2}(\nu)=\begin{pmatrix}
\mf{S}_n(\nu)^{-1}\\
-\mf{S}_n(\nu)^{-1}
\end{pmatrix}.
\]
By \ref{rc-bdd} we have $\bs{X}_n=\bs{X}_n(\nu)$ on $\Omega_n(\nu)$, while \eqref{rc-diag} is evident because $\Xi_n(\nu)\mf{C}_n(\nu)\Xi_n(\nu)^\top$ is the identity matrix in this case. 
So it remains to prove \eqref{eq-X}. Noting that $\mf{S}_n(\nu)$ is a diagonal matrix, \eqref{eq-X} follows from Corollary 15.80 in \cite{Janson1997} and Lemma \ref{acov-deriv} once we show that $\mf{S}_n(\nu)^{kk}\in\mathbb{D}_{2,\infty}$ for every $k=1,\dots,d^2$ and
\begin{align*}
&\sup_{n\in\mathbb{N}}\max_{1\leq k\leq d^2}\left(
\|\mf{S}_n(\nu)^{kk}\|_\infty
+\sup_{0\leq t\leq 1}\|D_t\mf{S}_n(\nu)^{kk}\|_{\infty,\ell_2}
+\sup_{0\leq s,t\leq 1}\|D_{s,t}\mf{S}_n(\nu)^{kk}\|_{\infty,\ell_2}
\right)<\infty.
\end{align*}
Since we can write $\mf{S}_n(\nu)^{kk}=(\mf{V}_n(\nu)^{kk})^{5/2}(\mf{V}_n(\nu)^{kk})^{-3}$, we obtain the desired result by combining Theorem 15.78 and Lemma 15.152 in \cite{Janson1997} with Lemma \ref{acov-deriv}. 
\end{proof}

\subsection{Proof of Lemma \ref{lemma:avar}}

We use the following notation: 
For a $d$-dimensional process $U=(U_t)_{t\in[0,1]}$, we set $\Delta^n_hU:=U_{h/n}-U_{(h-1)/n}$, $h=1,\dots,n$. 
Also, we set $\chi_h:=\vectorize[\Delta^n_hZ(\Delta^n_hZ)^\top]$ for $h=1,\dots,n$ and
\[
\wt{\mathfrak{C}}_n:=n\sum_{h=1}^n\chi_h\chi_h^\top-\frac{n}{2}\sum_{h=1}^{n-1}\left(\chi_h\chi_{h+1}^\top+\chi_{h+1}\chi_{h}^\top\right).
\]
\begin{lemma}\label{lemma:quarticity}
Assume \ref{rc-bdd}. Then $\sum_{h=1}^n(\|\Delta^n_hZ\|_{\ell_\infty}^4+\|\Delta^n_hX\|_{\ell_\infty}^4)=O_p(\log^2(d+r)/n)$ as $n\to\infty$. 
\end{lemma}

\begin{proof}
We use the same notation as in the proof of Lemma \ref{lemma:rc-max}. Then, we need to prove $\sum_{h=1}^n\|\Delta^n_h\bar{Z}\|_{\ell_\infty}^4=O_p(\log^2(d+r)/n)$ as $n\to\infty$. For every $\nu\in\mathbb{N}$ and $L>0$, we have
\begin{align*}
P\lpa\sum_{h=1}^n\|\Delta^n_h\bar{Z}\|_{\ell_\infty}^4>L\rpa
\leq P\lpa\sum_{h=1}^n\|\Delta^n_h\bar{Z}(\nu)\|_{\ell_\infty}^4>L\rpa
+P(\Omega_n(\nu)^c).
\end{align*}
Hence it suffices to prove $\sum_{h=1}^n\|\Delta^n_h\bar{Z}(\nu)\|_{\ell_\infty}^4=O_p(\log^2(d+r)/n)$ as $n\to\infty$ for any fixed $\nu\in\mathbb{N}$. 
By Lemma \ref{sharp-BDG} there is a universal constant $\mathsf{c}>0$ such that $\|\Delta^n_h\bar{M}(\nu)^j\|_p\leq\mathsf{c}\sqrt{p}\|\sqrt{\Delta^n_h[\bar{M}(\nu)^j,\bar{M}(\nu)^j]}\|_p$ for all $p\geq2$. Thus, by \ref{rc-bdd} we obtain $\|\Delta^n_h\bar{Z}(\nu)^j\|_p\leq C_\nu/n+\mathsf{c}\sqrt{C_\nu}\sqrt{p/n}$. Therefore, by \cite[Proposition 2.5.2]{Vershynin2018}, there is a constant $C'>0$ such that $\max_{j,h}\|\Delta^n_h\bar{Z}(\nu)^j\|_{\psi_2}\leq C'/\sqrt{n}$ for all $n$, where $\|\xi\|_{\psi_2}:=\inf\{\Lambda>0:\expe{\exp(|\xi|/\Lambda)}\leq2\}$ for a random variable $\xi$. Thus, \cite[Lemma 2.2.2]{VW1996} implies that there is a constant $C''>0$ such that $\max_{h}\|\|\Delta^n_h\bar{Z}(\nu)\|_{\ell_\infty}\|_{\psi_2}\leq C''\sqrt{\log(d+r)/n}$ for all $n$. Thus we obtain
\[
\ex{\sum_{h=1}^n\|\Delta^n_h\bar{Z}(\nu)\|_{\ell_\infty}^4}\leq4!^4C''\frac{\log^2(d+r)}{n},
\]
so the desired result follows from the Markov inequality. 
\end{proof}

\begin{lemma}\label{lemma:chihat}
Assume \ref{ass:vol}--\ref{ass:sparse.z} and \ref{rc-bdd}. Then $\sum_{h=1}^n\|\hat{\chi}_h-\chi_h\|_{\ell_\infty}^2=O_p(r^2(\log d)^3/n^2)$ as $n\to\infty$. 
\end{lemma}

\begin{proof}
Since $\hat{Z}_{h/n}=Z_{h/n}-(\hat{\beta}_n-\beta)X_{h/n}$, we have
\begin{align*}
\hat{\chi}_h-\chi_h
&=-\vectorize[(\hat{\beta}_n-\beta)\Delta^n_hX(\Delta^n_hZ)^\top]
-\vectorize[\Delta^n_hZ((\hat{\beta}_n-\beta)\Delta^n_hX)^\top]\\
&\quad+\vectorize[(\hat{\beta}_n-\beta)\Delta^n_hX((\hat{\beta}_n-\beta)\Delta^n_hX)^\top].
\end{align*}
Now, since $\|\vectorize(xy^\top)\|_{\ell_\infty}\leq\|x\|_{\ell_\infty}\|y\|_{\ell_\infty}$ for any $x,y\in\mathbb{R}^d$, it holds that
\begin{align*}
\|\hat{\chi}_h-\chi_h\|_{\ell_\infty}
&\leq2\|(\hat{\beta}_n-\beta)\Delta^n_hX\|_{\ell_\infty}\|\Delta^n_hZ\|_{\ell_\infty}
+\|(\hat{\beta}_n-\beta)\Delta^n_hX\|_{\ell_\infty}^2\\
&\leq2\opnorm{\hat{\beta}_n-\beta}_\infty\|\Delta^n_hX\|_{\ell_\infty}\|\Delta^n_hZ\|_{\ell_\infty}
+\opnorm{\hat{\beta}_n-\beta}_\infty^2\|\Delta^n_hX\|_{\ell_\infty}^2.
\end{align*}
Therefore, we obtain
\begin{align*}
\sum_{h=1}^n\|\hat{\chi}_h-\chi_h\|_{\ell_\infty}^2
&\leq2\opnorm{\hat{\beta}_n-\beta}_\infty^2\sum_{h=1}^n(\|\Delta^n_hX\|_{\ell_\infty}^4+\|\Delta^n_hZ\|_{\ell_\infty}^4)
+2\opnorm{\hat{\beta}_n-\beta}_\infty^4\sum_{h=1}^n\|\Delta^n_hX\|_{\ell_\infty}^4.
\end{align*}
Now, noting Lemma \ref{lemma:rc-max}, we infer that $\opnorm{\hat{\beta}_n-\beta}_\infty=O_p(r\sqrt{(\log d)/n})$ from the proof of Lemma \ref{lemma:factor-cov}\ref{opi-beta}. Thus, we complete the proof by Lemma \ref{lemma:quarticity}. 
\end{proof}

\begin{proof}[Proof of Lemma \ref{lemma:avar}]
Since $\|\wt{\mf{C}}_n-\mf{C}_n\|_{\ell_\infty}=O_p((\log d)^{2}/\sqrt{n}+n^{-\gamma})$ by Proposition 4.1 in \cite{Koike2018sk}, it suffices to prove $\|\hat{\mf{C}}_n-\wt{\mf{C}}_n\|_{\ell_\infty}=O_p(r(\log d)^{5/2}/\sqrt{n})$. Since $\|\vectorize(xy^\top)\|_{\ell_\infty}\leq\|x\|_{\ell_\infty}\|y\|_{\ell_\infty}$ for any $x,y\in\mathbb{R}^d$, Lemma \ref{lemma:quarticity} yields $\sum_{h=1}^n\|\chi_h\|_{\ell_\infty}^2=O_p((\log d)^2/n)$. Combining this with Lemma \ref{lemma:chihat} and $r^2(\log d)/n=O(1)$, we also obtain $\sum_{h=1}^n\|\hat{\chi}_h\|_{\ell_\infty}^2=O_p((\log d)^2/n)$. Now the desired result follows from the Schwarz inequality and Lemma \ref{lemma:chihat}. 
\end{proof}

\subsection{Proof of Corollary \ref{coro:feasible}}

\ref{feasible-a}
Since $\opnorm{\Theta_Z}_\infty=O_p(1)$ by \eqref{Theta-bdd} and \ref{ass:adjacent}, we have $\|\mf{C}_n\|_{\ell_\infty}+\|\mf{V}_n\|_{\ell_\infty}=O_p(1)$ by \ref{rc-bdd} and $\lambda_n^{-1}\opnorm{\hat{\Theta}_{Z,\lambda_n}\otimes\hat{\Theta}_{Z,\lambda_n}-\Theta_Z\otimes\Theta_Z}_\infty=O_p(s_n)$ by Theorem \ref{rc:rate}. Combining this with Lemma \ref{lemma:avar} and assumption, we obtain $\|\hat{\mf{V}}_n\|_{\ell_\infty}=O_p(1)$ and $(\log d)\|\hat{\mf{V}}_n-\mf{V}_n\|_{\ell_\infty}\to^p0$. 
Noting \eqref{Theta-bdd} and the fact that $\hat{\mf{S}}_n$ is a diagonal matrix, we also obtain $\opnorm{\hat{\mf{S}}_n^{-1}}_{\infty}=O_p(1)$ and $(\log d)\opnorm{\hat{\mf{S}}_n^{-1}-\mf{S}_n^{-1}}_\infty\to^p0$. Since \eqref{AMN:non-student} yields $\|\sqrt{n}\vectorize(\hat{\Theta}_{Z,\lambda_n}-\Gamma_{Z,n}-\Theta_Z)\|_{\ell_\infty}=O_p(\sqrt{\log d})$, we obtain $\sqrt{\log d}\|\sqrt{n}(\hat{\mf{S}}_n^{-1}-\mf{S}_n)\vectorize(\hat{\Theta}_{Z,\lambda_n}-\Gamma_{Z,n}-\Theta_Z)\|_{\ell_\infty}\to^p0$. Now the desired result follows from Theorem \ref{rc:AMN} and \cite[Lemma 3.1]{Koike2018sk}.

\ref{feasible-b}
The same argument as above implies that $(\log d)^2\|\hat{\mf{V}}_n-\mf{V}_n\|_{\ell_\infty}\to^p0$ and $(\log d)^2\opnorm{\hat{\mf{S}}_n^{-1}-\mf{S}_n^{-1}}_\infty\to^p0$. Thus, the desired result follows from \cite[Proposition 3.1]{Koike2018sk}. \hfil\qed

\section{Proof of Lemma \ref{rc:concentrate}}\label{proof:concentrate}

In this appendix we prove Lemma \ref{rc:concentrate} with the help of two general martingale inequalities. 
The first one is the Burkholder-Davis-Gundy inequality with a sharp constant:
\begin{lemma}[\citet{BY1982}, Proposition 4.2]\label{sharp-BDG}
There is a universal constant $\mathsf{c}>0$ such that
\[
\left\|\sup_{0\leq t\leq T}|M_t|\right\|_p\leq \mathsf{c}\sqrt{p}\left\|[M,M]_T^{1/2}\right\|_p
\]
for any $p\in[2,\infty)$ and any continuous martingale $M=(M_t)_{t\in[0,T]}$ with $M_0=0$. 
\end{lemma}

The second one is a Bernstein-type inequality for martingales:
\begin{lemma}\label{bernstein}
Let $(\xi_i)_{i=1}^n$ be a martingale difference sequence with respect to the filtration $(\mcl{G}_i)_{i=0}^n$. 
Suppose that there are constants $a,b>0$ such that $\sum_{i=1}^n\expe{|\xi_i|^k\mid\mcl{G}_{i-1}}\leq k!a^{k-2}b^2/2$ a.s.~for any integer $k\geq 2$. Then, for any $x\geq0$,
\[
P\lpa\max_{1\leq m\leq n}\labs\sum_{i=1}^m\xi_i\rabs\geq x\rpa\leq2\exp\lpa-\frac{x^2}{b^2+b\sqrt{b^2+2ax}}\rpa.
\]
\end{lemma}

\begin{proof}
This is a special case of \citet[Theorem 3.3]{Pinelis1994}. In fact, since $\mathbb{R}$ is a Hilbert space, we can apply this result with $\mcl{X}=\mathbb{R}$ and $D=1$ in the notation of that paper.  
\end{proof}

\begin{proof}[Proof of Lemma \ref{rc:concentrate}]
For every $h=1,\dots,n$, set
\[
\xi_{n,h}:=\sqrt{n}\left\{\int_{t_{h-1}}^{t_h}(M_t-M_{t_{h-1}})dN_t+\int_{t_{h-1}}^{t_h}(N_t-N_{t_{h-1}})dM_t\right\}.
\]
It\^o's formula yields
\[
\sqrt{n}\lpa\wh{[M,N]}^n_1-[M,N]_1\rpa=\sum_{h=1}^n\xi_{n,h}.
\]
Also, by assumption $(\xi_{n,h})_{h=1}^n$ is a martingale difference with respect to $(\mathcal{F}_{t_h})_{h=0}^n$. Moreover, for any integer $k\geq2$, we have
\begin{align*}
&\expe{|\xi_{n,h}|^k\mid\mcl{F}_{t_{h-1}}}\\
&\leq2^{k-1}n^{k/2}\ex{\labs\int_{t_{h-1}}^{t_h}(M_t-M_{t_{h-1}})dN_t\rabs^{k}+\labs\int_{t_{h-1}}^{t_h}(N_t-N_{t_{h-1}})dM_t\rabs^{k}\mid\mcl{F}_{t_{h-1}}}
\\
&\leq 2^{k-1}n^{k/2}\mathsf{c}^kk^{k/2}\ex{\lpa\int_{t_{h-1}}^{t_h}(M_t-M_{t_{h-1}})^2d[N, N]_t\rpa^{k/2}+\lpa\int_{t_{h-1}}^{t_h}(N_t-N_{t_{h-1}})^2d[M, M]_t\rpa^{k/2}\mid\mcl{F}_{t_{h-1}}}\\
&\pushright{(\because\text{Lemma \ref{sharp-BDG}})}\\
&\leq 2^{k-1}\mathsf{c}^kk^{k/2}L^{k/2}\ex{\sup_{t_{h-1}<t\leq t_h}|M_t-M_{t_{h-1}}|^{k}+\sup_{t_{h-1}<t\leq t_h}|N_t-N_{t_{h-1}}|^k\mid\mcl{F}_{t_{h-1}}}
~(\because\text{\eqref{eq:lipschitz}})\\
&\leq 2^{k-1}\mathsf{c}^{2k}k^{k}L^{k/2}\ex{([M, M]_{t_h}-[M, M]_{t_{h-1}})^{k/2}+([N, N]_{t_h}-[N, N]_{t_{h-1}})^{k/2}\mid\mcl{F}_{t_{h-1}}}
~(\because\text{Lemma \ref{sharp-BDG}})\\
&\leq 2^{k}\mathsf{c}^{2k}k^{k}\frac{L^{k}}{n^{k/2}}
~(\because\text{\eqref{eq:lipschitz}}),
\end{align*}
where $\mathsf{c}>0$ is the universal constant appearing in Lemma \ref{sharp-BDG}. 
Thus, using Stirling's formula, we obtain
\[
\sum_{h=1}^n\expe{|\xi_{n,h}|^k\mid\mcl{F}_{t_{h-1}}}\leq 2^{k}\mathsf{c}^{2k}\frac{e^k}{\sqrt{2\pi k}}k!\frac{L^{k}}{n^{k/2-1}}
\leq\frac{k!}{2}\lpa\frac{a_0}{\sqrt{n}}\rpa^{k-2}b_0^2,
\]
where $a_0:=2e\mathsf{c}^2L$ and $b_0:=2\sqrt{2}\mathsf{c}^2Le/(2\pi)^{1/4}$. 
Hence, Lemma \ref{bernstein} yields
\[
P\lpa\labs\sum_{h=1}^n\xi_{n,h}\rabs\geq x\rpa\leq2\exp\lpa-\frac{x^2}{b_0^2+b_0\sqrt{b_0^2+2(a_0/\sqrt{n})x}}\rpa
\]
for every $x\geq0$. Consequently, when $x\in[0,\theta\sqrt{n}]$ for some $\theta>0$, we have
\[
P\lpa\labs\sum_{h=1}^n\xi_{n,h}\rabs\geq x\rpa\leq2\exp\lpa-C_{L,\theta}x^2\rpa
\]
with $C_{L,\theta}:=(b_0^2+b_0\sqrt{b_0^2+2a_0\theta})^{-1}$. This completes the proof.
\end{proof}

\section*{Acknowledgements}

This work was supported by JST CREST Grant Number JPMJCR14D7 and JSPS KAKENHI Grant Numbers JP17H01100, JP18H00836, JP19K13668.

{\small
\renewcommand*{\baselinestretch}{1}\selectfont
\addcontentsline{toc}{section}{References}

}

\end{document}